\documentclass[11pt]{article} 
\usepackage[T1]{fontenc}    
\usepackage{hyperref}       
\usepackage[table]{xcolor}         

\usepackage{mathbbol,bbm}  
\usepackage{amsfonts,amsmath,amssymb,amsthm}   
\usepackage{cmbright}

\usepackage{lipsum}

\DeclareMathOperator*{\argmin}{\mbox{argmin}}
 
\def\prox{\mbox{prox}}
\usepackage{tikz}
\usepackage{setspace}

\usepackage{fancyhdr}

\usepackage{graphicx}
\usepackage{subfig}
\newtheorem{theorem}{Theorem}
\newtheorem{proposition}{Proposition}
\newtheorem{lemma}{Lemma}[section]
 
\newtheorem{remark}{Remark}
\usepackage{multirow}

\def\bbR{\mathbb{R}}
\def\bbN{\mathbb{N}}
\def\sC{\mathcal{C}}
\def\sP{\mathcal{P}}
\def\sO{\mathcal{O}}

\newcommand{\II}{\mbox{I}}

\usepackage{xcolor}

\definecolor{TypalBlue}{HTML}{7086DF}  
\definecolor{TypalBlueDark}{HTML}{5467b2}
\definecolor{TypalRed}{HTML}{D32B08}    
\definecolor{TypalGreen}{HTML}{3cb371}

\usepackage{enumitem}
\usepackage{geometry}
\author{%
  Howard Heaton \\
  Typal Academy  
}

\hypersetup{
    pdftitle={Proximal Projection Method for Stable Linearly Constrained Optimization},  
    pdfauthor={Howard Heaton},    
    pdfsubject={optimization},
    pdfcreator={Howard Heaton}, 
    pdfproducer={LaTex},  
    pdfkeywords={proximal, projection, convex optimization, Douglas-Rachford splitting, basis pursuit, compressed sensing, matrix completion, principal component pursuit, earth mover's distance, Wasserstein distance},
    colorlinks=true,       
    linkcolor=TypalBlueDark,          
    citecolor=TypalBlueDark,        
    filecolor=TypalBlueDark,      
    urlcolor=TypalBlueDark           
}
 
\def\title{Proximal Projection Method for Stable Linearly Constrained Optimization}

\author{%
  Howard Heaton \\
  Typal Academy   
}
 
\renewcommand{\maketitle}{ 
    \begin{center}
        \rule{\textwidth}{4pt}\\[20pt]
        {\LARGE \title} \\[10pt]
        \rule{\textwidth}{2pt} \\[20pt]

        {\bf Howard Heaton} \\ 
        Typal Academy \\[20pt] 
    \end{center} 
}

\usepackage{algorithm}
\usepackage{algorithmic} 

\def\directory{.}

\def\bpTimePP{8.06}
\def\bpTimeLB{7.40}
\def\bpTimePDHG{20.60}
\def\bpTimeLMM{22.85}
\def\emdTimePP{107.20}\def\emdTimeGPDHG{107.38}\def\emdTimePDHG{43.71}\def\emdTimePPoverPDHG{2.5}

\begin{document}


\maketitle
 
\begin{abstract}
    Many applications using large datasets require efficient methods for minimizing a proximable convex function subject to satisfying a set of linear constraints within a specified tolerance.  
    For this task,  we present a proximal projection (PP) algorithm, which is an instance of Douglas-Rachford splitting that directly uses projections onto the set of constraints. Formal guarantees are presented to prove  convergence of PP estimates to optimizers. Unlike many methods  that obtain feasibility asymptotically,  each PP iterate is feasible. Numerically, we show PP  either matches or outperforms  alternatives (\textit{e.g.} linearized Bregman, primal dual hybrid gradient, proximal augmented Lagrangian, proximal gradient) on problems in basis pursuit, stable matrix completion, stable principal component pursuit, and the computation of earth mover's distances. 
\end{abstract}
  
\vspace*{10pt}
{\small
    \begin{tabular}{rp{6.05in}}
    {\bf Key words:}\hspace*{-5pt}
    & proximal, projection, constrained optimization, Douglas-Rachford splitting, basis pursuit, compressed sensing, matrix completion, principal component pursuit, earth mover's distance, Wasserstein distance 
    \end{tabular}
}
 
\section{Introduction}

With the profound increase in the use of extremely high-dimensional data, an ongoing challenge is to create efficient tools for processing this data. In several important applications, this processing   takes the form of solving a convex optimization problem. 
Practical value is, thus, derived from efficient and easy-to-code methods for such tasks, which is the focus of this work.  
Specifically, for a convex  function\footnote{We set $\overline{\bbR} \triangleq \bbR \cup \{+\infty\}$.} $f\colon\bbR^n\rightarrow\overline{\bbR}$, a matrix $A\in\bbR^{m\times n}$,  a vector $b\in\bbR^m$ and a scalar $\varepsilon \geq 0$, this work considers the problem  
\begin{equation}
    \min_{x\in\bbR^n} f(x) 
    \quad\mbox{s.t.}\quad 
    \|Ax - b\|\leq \varepsilon.
    \tag{P}
    \label{eq: problem}
\end{equation} 
We let $\sC \triangleq \{ x : \|Ax-b\|\leq\varepsilon\}$ denote the constraint set so that (\ref{eq: problem}) equates to minimizing $f$ over $\sC$. 
As many applications introduce a tolerance $\varepsilon > 0$ on linear constraints to ensure stability with respect data being corrupted by noise, herein we refer to (\ref{eq: problem}) as a stable linearly constrained optimization problem.
\\ 

The problem (\ref{eq: problem}) is well-studied.
When $\varepsilon = 0$, several algorithms can be directly applied (\textit{e.g.}  primal dual hybrid gradient \cite{chambolle2011first}, conditional gradient \cite{frank1956algorithm}). When $\varepsilon > 0$, an auxiliary variable can be introduced to decompose $\sC$ into a linear constraint and a ball constraint; for example, the alternating direction method of multipliers \cite{boyd2011distributed,deng2016global} can be readily applied to that formulation. In some works, the constraint is moved into the objective as a quadratic penalty \cite{zhou2010stable,candes2010matrix}; for an appropriate penalty weight, this soft-penalty variation shares the same minimizers as (\ref{eq: problem}). The superiorization methodology \cite{censor2014projected,censor2010perturbation} may also approximate solutions to (\ref{eq: problem}) by interweaving projection steps onto $\sC$ (or sets whose intersection forms $\sC$) with subgradient steps. Other approaches \cite{aybat2014efficient,cai2013fast} use various forms of smoothing or added regularization to solve problems that approximate (\ref{eq: problem}).
\\

Much prior work aims to solve (\ref{eq: problem}) by either approximating (\ref{eq: problem}) or obtaining feasibility asymptotically. In contrast, we solve (\ref{eq: problem}) and maintain feasibility at each step of our iterative algorithm. This is done at comparable per-iteration cost to existing methods by  interweaving projections onto $\sC$ and proximal operations (defined below). 

\newpage

\paragraph{Contribution.} 
Our main result is Algorithm \ref{alg: pp} and its convergence to solutions of (\ref{eq: problem}), which is possible due to a novel formula for the projection onto the constraint set $\sC$. Our numerical examples show favorable performance of PP against alternatives in a varied collection of practical problems, \textit{i.e.} basis pursuit, stable principal component pursuit, computation of earth mover's distances, and stable matrix completion.

\paragraph{Notation.} 
Here, $\|\cdot\|$ is the Euclidean norm,
$\|\cdot\|_F$ is the Frobenius norm, $\|\cdot\|_1$ is the $1$-norm, and $\|\cdot\|_\star$ is the nuclear norm. The relative interior domain of $f$ is $\mbox{ri}(\mbox{dom}(f))$. For an integer $n\in\bbN$, we set $[n] \triangleq \{1,\ldots,n\}$.

\begin{algorithm}[t]
    \caption{{\tt Proximal Projection (PP) for (\ref{eq: problem}) with $\varepsilon=0$}}
    \label{alg: pp-eps=0}
    \setstretch{1.35}
    \def\tab{\hspace*{15pt}}
    \begin{algorithmic}[1]
        \STATE{ 
            PP($f$, $A$, $b$):
        }
        \STATE{
            \tab {\bf initialize}   iterate $z \in\bbR^n$ and parameter $\alpha > 0$ 
        }

        \STATE{
            \tab {\bf while} stopping criteria not met
        }
             
        \STATE{
            \tab\tab $x\leftarrow z - A^\top (AA^\top)^{-1} (Az - b)$
        }
        \STATE{
            \tab\tab $z \leftarrow z + \mbox{prox}_{\alpha f}\left( 2x - z\right)  -x$
        }

        \STATE{
            \tab {\bf return} $x$
        }
    \end{algorithmic}
    \end{algorithm}  

\section{Main Results} \label{sec: results}
In this section, we define our proposed proximal  projection (PP) algorithm and analytically show it generates a solution to (\ref{eq: problem}). Throughout, we make use of combinations of the following conditions: 
 
\begin{enumerate}[label=(C\arabic*)]
    \item the function $f\colon\bbR^n\rightarrow \overline{\bbR}$ is closed, convex, and proper; \label{c: CCP}
    \item either the matrix $A$ has full row-rank or $\varepsilon > 0$;
    \label{c: full-rank}  
    \item   
    there is $y \in \bbR^n$ such that, if $\varepsilon = 0$, then $Ay = b$ and, if $\varepsilon > 0$, then  $\|Ay-b\| < \varepsilon$;
    \label{c: b-image}    
    \item either $f$ is coercive or $\sC$ is bounded; \label{c: coercive-bounded} 
    \item condition \ref{c: b-image} holds for $y \in \mbox{ri}(\mbox{dom}(f))$.
    \label{c: proper-feasible}
\end{enumerate}

To minimize the function $f$,   we make use of a proximal operator. 
For a parameter $\alpha > 0$, this is defined by  
\begin{equation}
    \prox_{\alpha f}(x) \triangleq \argmin_{u \in \bbR^n} \alpha f(u) + \frac{1}{2}\|u - x \|^2,
\end{equation}
and \ref{c: CCP} ensures it uniquely exists \cite{beck2017first,bauschke2017convex,ryu2022large}. In many instances, explicit formulas exist for the proximal (\textit{e.g.} see \cite[Chapter 6]{beck2017first}). Proximals also generalize projections. Specifically, if $\sC$ is a nonempty, closed and convex set and $\delta_\sC$ is the indicator function taking value $0$ in $\sC$ and $\infty$ elsewhere, then 
\begin{equation}  
    \prox_{\alpha \delta_\sC}(x)  
    = \sP_{\sC}(x) 
    \triangleq \argmin_{u\in\sC} \frac{1}{2}\|u-x\|^2.
\end{equation}
That is, the proximal for $\delta_\sC$ is precisely the projection $\sP_\sC$ onto $\sC$.
This leads to our next conditions \ref{c: full-rank} and \ref{c: b-image}, which are used to obtain our projection formula in the following lemma. (See Appendix \ref{app: proofs} for a proof.)

\begin{algorithm}[t]
    \caption{{\tt Proximal   Projection (PP)  for (\ref{eq: problem})   }}
    \label{alg: pp}
    \setstretch{1.25}
    \def\tab{\hspace*{15pt}}
    \begin{algorithmic}[1]
        \STATE{
            PP($f$, $A$, $b$, $\varepsilon$):
        }
        \STATE{
            \tab {\bf initialize} iterate $z \in\bbR^n$ and parameter $\alpha > 0$
        }
        \STATE{
            \tab {\bf while} stopping criteria not met
            \label{alg: pp-while-start}
        }
        \STATE{
            \tab\tab {\bf if} $\|Az-b\| \leq \varepsilon$
        }
        \STATE{\tab\tab\tab$x \leftarrow z$}         
        \STATE{\tab\tab{\bf else}}
        \STATE{
            \tab\tab\tab$\tau \leftarrow \mbox{solution}(1 = \tau\|(AA^\top + \varepsilon \tau \mbox{I})^{-1} (Az-b)\|)$
            \label{alg: pp-tau}
        }
        \STATE{
            \tab\tab\tab$x\leftarrow z - A^\top (AA^\top + \varepsilon\tau \II)^{-1} (Az-b)$
            \label{alg: pp-x-update}
        }

        \STATE{
            \tab\tab $z \leftarrow z+ \mbox{prox}_{\alpha f}\left( 2x - z\right) - x$
            \label{alg: pp-z-update}
        }
        \STATE{
            \tab {\bf return} $x$
        } 
    \end{algorithmic}
    \end{algorithm}

\def\projectionProposition{
    If conditions \ref{c: full-rank} and  \ref{c: b-image} hold and $\sC \triangleq \{ x : \|Ax-b\|\leq\varepsilon\}$, then 
    \begin{equation}
        \sP_{\sC}(x) 
        = \begin{cases}
        \begin{array}{cl}
            x & \mbox{if}\ \|Ax-b\|\leq\varepsilon, \\
            x - A^\top(AA^\top +\varepsilon\tau_x \II)^{-1}(Ax-b) & \mbox{otherwise},
        \end{array}
        \end{cases}
    \end{equation}
    where, if $\|Ax-b\|>\varepsilon$, the scalar $\tau_x$ is the unique positive solution to 
    \begin{equation}
        1 = \tau \| (AA^\top +\varepsilon \tau \II)^{-1} (Ax-b) \|.
    \end{equation}
}

\begin{proposition}[Projection Formula] \label{prop: projection}
    \projectionProposition
\end{proposition}

The final conditions \ref{c: coercive-bounded} and \ref{c: proper-feasible} ensure a solution exists and total duality holds, a key condition required by many operator splitting methods to establish convergence \cite{ryu2022large}. To describe our method, let $z^1 \in \bbR^n$ and $\alpha > 0$. We  construct sequences  $\{x^k\}$ and $\{z^k\}$ with the update at each index $k$ given by the  formulas
\begin{subequations}
    \begin{align} 
        x^{k} & = \begin{cases}\begin{array}{cl}
            z^k & \mbox{if}\ \|Az^k - b \|\leq \varepsilon, \\ 
            z^k - A^\top (AA^\top + \varepsilon\tau_{z^k} \II)^{-1}(Az^k - b) & \mbox{otherwise,}
        \end{array}\end{cases} \\ 
        z^{k+1} & = z^k + \prox_{\alpha f}(2x^{k} - z^k) - x^{k},
    \end{align}\label{eq: pp-iteration}\end{subequations}with $\tau_{z^k}$ defined as in Proposition \ref{prop: projection}.
    The iteration in (\ref{eq: pp-iteration}) is a special case of a more general scheme known as Douglas-Rachford splitting (DRS) \cite{eckstein1992douglas,lions1979splitting}, which has many uses (\textit{e.g.} finding the zero of a sum of monotone operators \cite{eckstein1992douglas,eckstein1989splitting},  feasibility problems \cite{lindstrom2021survey}, combinatorial optimization \cite{aragon2014recent}). Making use of prior DRS results, the following theorem justifies use of Algorithm \ref{alg: pp-eps=0} (the $\varepsilon = 0$ case) and Algorithm \ref{alg: pp} (the $\varepsilon\geq 0$ case).

\def\mainTheorem{
    If conditions \ref{c: CCP} to \ref{c: proper-feasible} hold, then the sequences $\{x^k\}$ and $\{z^k\}$ generated by (\ref{eq: pp-iteration}) converge, with $\{x^k\}$ converging to a solution of (\ref{eq: problem}). Moreover, $\|Ax^k - b\|\leq\varepsilon$ for all $k$.
}

\begin{theorem}[Convergence of PP]\label{thm: main-theorem} 
    \mainTheorem
\end{theorem}

A proof of Theorem \ref{thm: main-theorem} is provided in Appendix \ref{app: proofs}. The rest of this section considers per-iteration costs of PP.

\begin{remark}[Computation of $\tau$]
    Explicit formulas can sometimes be derived for $\tau_x$ in Line \ref{alg: pp-tau} of Algorithm \ref{alg: pp} (\textit{e.g.} see Subsections \ref{subsec: spcp} and \ref{subsec: mc}).
    Otherwise, $\tau_x$ may be computed via a 1D solver (\textit{e.g.} bisection method).
    Note $0 \leq \tau_x \leq  \sigma_{max}(A)^2/(\|Ax-b\|-\varepsilon)$, with $\sigma_{max}(A)$ the largest singular value of $A$ (see  Lemma \ref{lemma: tau-exists}).
\end{remark}

The update formulas for PP and standard alternatives involve a proximal operation and matrix multiplications. The  costs  associated with multiplication by $A$ are $\sO(mn^2)$ and for $A^\top$ they are $\sO(m^2n)$.
Unlike most schemes we compare to, PP includes multiplication by $(AA^\top + \varepsilon\tau_x \II)^{-1}\in\bbR^{m\times m}$. In the worst case, this adds $\sO(m^3)$ cost.
Computing the matrix $(AA^\top + \varepsilon\tau_x \II)^{-1}$ can also  add   $\sO(m^3)$ cost. However, with certain structures of $A$, both of these costs can be reduced to $\sO(m)$, yielding little impact. 
The numerical examples below investigate whether the combination of per-iteration cost and convergence rate of PP is more efficient than alternatives.

\begin{remark}[Inversion with $\varepsilon=0$]
    In the case that $\varepsilon = 0$, there is a one-time computation of $(AA^\top)^{-1}$ that has cost $\sO(m^3)$. Ammortizing this over hundreds of iterations yields negligible per-iteration cost. Moreover, in some applications the same matrix $A$ is repeatedly used with new measurement data $b$, in which case $(AA^\top)^{-1}$ may be computed in an offline setting. 
\end{remark}

\begin{remark}[SVD Inverse]
    If $USV^\top$ is the singular value decomposition (SVD) of $A$, with $\Sigma = \mbox{diag}(\sigma_i)$, then  
    \begin{equation} 
        A^\top (AA^\top + \varepsilon\tau \II)^{-1} = V\, \mbox{diag}\left(\frac{\sigma_i}{\sigma_i^2 + \varepsilon\tau}\right) U^\top.
    \end{equation}
    Thus, if the SVD of $A$ is available, then the matrix inverse in Line \ref{alg: pp-x-update} of Algorithm \ref{alg: pp} is readily available too.
\end{remark}

\begin{remark}[Inversion of Tridiagonal Matrices] When the matrix $A$ has tridiagonal structure, Thomas' algorithm \cite{thomas1949elliptic} can be used to multiply $(AA^\top)^{-1}$ with $\sO(m)$ cost rather than the $\sO(m^3)$ cost via Gaussian elimination.  
    \label{remark: thomas-alg}
\end{remark}

\section{Numerical Examples}

\def\placeholder{
    \begin{tikzpicture}
        \fill[black!20!white] (0,0) rectangle++ (5,3) node[midway] {placeholder};
    \end{tikzpicture} 
} 

We provide four numerical examples. In each setting, conditions \ref{c: CCP} to \ref{c: proper-feasible} hold and, when applicable,  simplifications of Algorithm \ref{alg: pp} and Proposition \ref{prop: projection} are presented. Shown methods may be accelerated (including PP) to yield better results than shown; for simplicity of comparison, we restrict attention to unaccelerated variants.
Code\footnote{See Python source code on Github: \href{https://github.com/TypalAcademy/proximal-projection-algorithm}{github.com/TypalAcademy/proximal-projection-algorithm}.}  was run on a Macbook with an Apple M1 Pro chip and 16 GB RAM. \\[-22pt]

\subsection{Basis Pursuit}

In the field of compressed sensing, the aim is to recover a sparse signal $x^\star \in\bbR^n$ via a collection of linear measurements $b\in\bbR^m$ (see the survey \cite{qaisar2013compressive}).  If $m < n$ and the matrix $A\in\bbR^{m\times n}$ defining the measurements satisfies certain conditions (\textit{e.g.} restricted isometry \cite{candes2005decoding,candes2006near}) the signal $x^\star$ is often the solution to the problem \\[-8pt]
\begin{equation}
    \min_{x\in\bbR^{n}} \|x\|_1 
    \quad \mbox{s.t.}\quad   
    Ax = b.
    \label{eq: problem-bp}
    \tag{BP} 
\end{equation}
Here we set the matrix $A$ to have i.i.d. Gaussian entries, $m=500$, and $n=2000$. Elements of $x^\star$ are independently nonzero with probability $p=0.05$ and the nonzero values are i.i.d. Gaussian. Algorithm \ref{alg: pp-eps=0} is used with $f(x) = \|x\|_1$, for which \\[-13pt]
\begin{equation} 
    \prox_{\alpha \|\cdot\|_1}(x) 
    = \mbox{shrink}(x,\ \alpha)
    \triangleq \mbox{sgn}(x) \odot \max\{ |x| - \alpha, 0  \},
    \label{eq: shrink-1}
\end{equation} 
where $\odot$ is the element-wise product and $\mbox{sgn}$ is the sign function with value $1$ for positive input, $-1$ for negative input, and $0$ otherwise. 
Further setup details are described in Appendix \ref{app: bp}.\\[-5pt]
 
We compare PP against similar first-order methods: Linearized Bregman (LB) \cite{cai2009linearized,yin2010analysis,osher2010fast}, Primal Dual Hybrid Gradient (PDHG) \cite{chambolle2011first}, and the Linearized Method of Multipliers\footnote{This is also referred to as an instance of the proximal augmented Lagrangian method.} (LMM) \cite{ryu2022large}.  
Figure \ref{fig: plots-bp}a shows PP is feasible at each iteration (to machine precision) while the benchmark methods reduce iteratively reduce constraint violation. The other plots in Figure \ref{fig: plots-bp}b and \ref{fig: plots-bp}c show PP linearly converges to machine precision within hundreds of iterations while the other methods exhibit sublinear convergence. The mean time for 10 trials of PP, LB, PDHG, and LMM to compute 2,000 iterations were, respectively, \bpTimePP s, \bpTimeLB s, \bpTimePDHG s, and \bpTimeLMM s. Specifically, PP has comparable per-iteration cost to LB and less per-iteration cost than PDHG and LMM.  Thus, in this example, PP is fastest (converging in finite steps) and the only method to achieve feasibility. 
  
\begin{figure}[H] 
    \centering 
    \subfloat[Violation $\|Ax^k-b\|$]{\includegraphics[page=1]{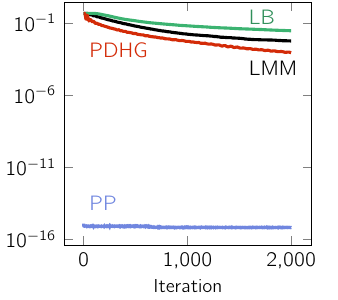}}       
    \subfloat[Objective $\|x^k\|_1$]{\includegraphics[page=2]{./images/basis-pursuit/bp-plots.pdf}}
    \subfloat[Residual $\|x^{k+1}-x^k\|$]{\includegraphics[page=3]{./images/basis-pursuit/bp-plots.pdf}}     
 
    \caption{ 
    Basis pursuit (\ref{eq: problem-bp}) plots, 
    using median for 10 samples of $A$ and $x^\star$.
    Each PP iterate is feasible, as shown in (a). Note LB and PDHG update other variables for several steps before $\|x^{k+1}-x^k\|\neq 0$.}
    \label{fig: plots-bp}
\end{figure}

\newpage

\begin{figure}[t] 
    \centering
    \subfloat[Original Image $M$]{\includegraphics[width=0.32\textwidth]{\directory/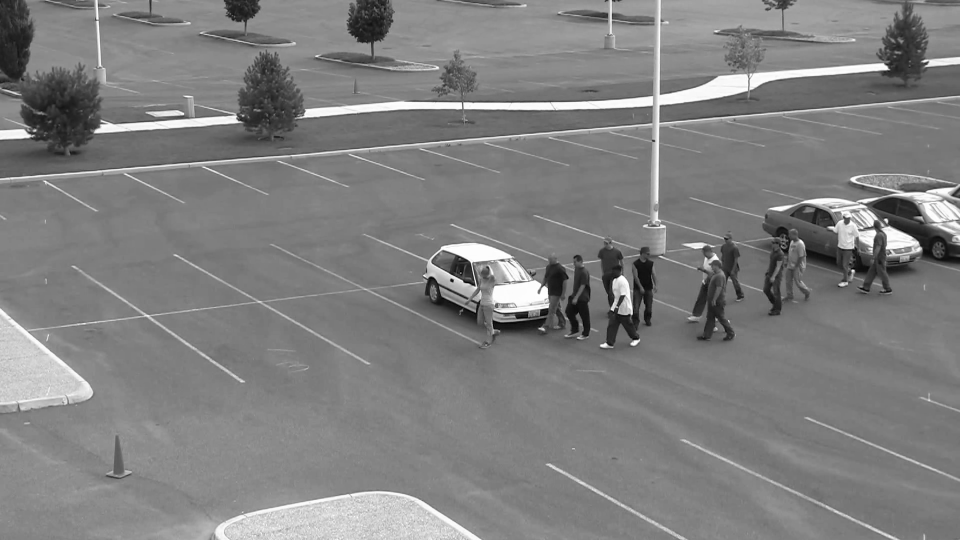}}
    \hspace*{5pt}
    \subfloat[Low Rank $L$]{\includegraphics[width=0.32\textwidth]{\directory/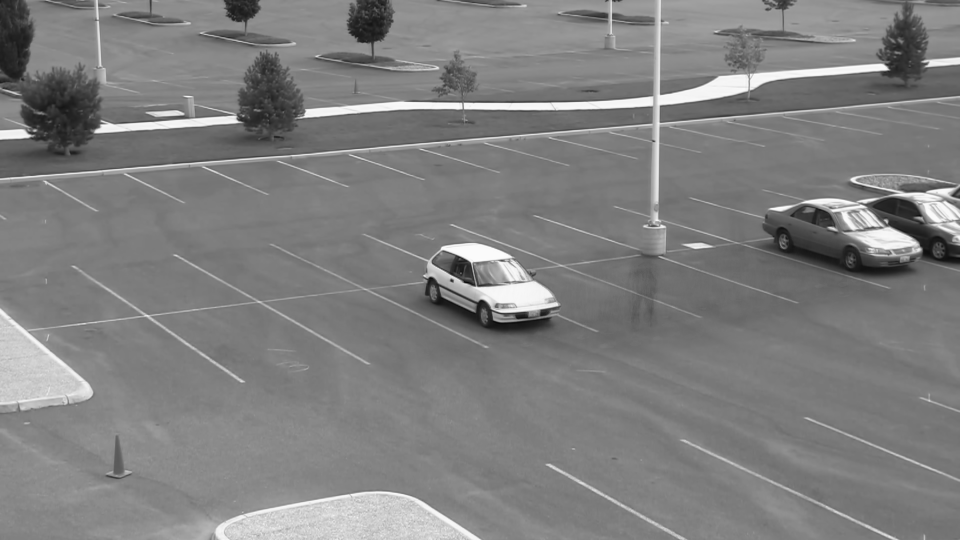}}
    \hspace*{5pt}
    \subfloat[Sparse $S$]{\includegraphics[width=0.32\textwidth]{\directory/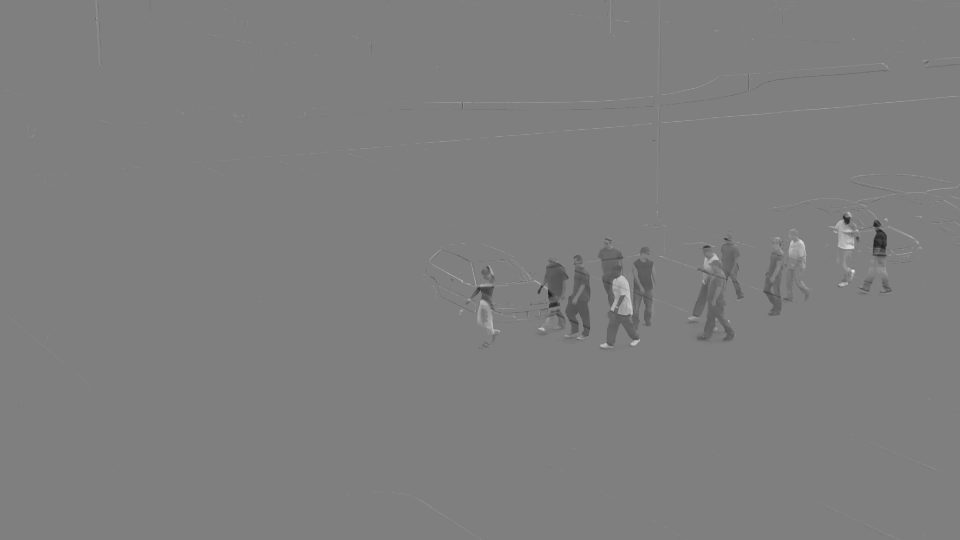}}
    \caption{Example output from PP  (Algorithm \ref{alg: pp-spcp}) with $M$ as a video consisting of 250 grayscale 960$\times$540 images (each image forming a column of $M$). The video shows a group of people entering from the right and walking to the left. The matrix $M$ is approximately represented as the sum  of a low rank matrix $L$ (background) and a sparse matrix $S$ (people walking).}
    \label{fig: PCP-images} 
    \vspace*{10pt}
\end{figure}

\subsection{Stable Principal Component Pursuit}  \label{subsec: spcp}

Many applications with high dimensional data (\textit{e.g.} image alignment \cite{peng2012rasl},   hyperspectral image restoration \cite{zhang2013hyperspectral}, scene triangulation \cite{zhang2012tilt}) consider the problem of recovering a low-rank matrix (\textit{i.e.} the principal components) from a high-dimensional matrix that has sparse errors and entry-wise noise.  
For a matrix $M\in \bbR^{n_1\times n_2}$, the task we consider is finding a corresponding low-rank matrix $L \in \bbR^{n_1\times n_2}$ and a sparse matrix $S \in \bbR^{n_1\times n_2}$
such that 
\begin{equation}
    M = L + S + N, 
\end{equation}
where $N$ is a noise term, for which we assume $\|N\|_F \leq \varepsilon$.
Following \cite{zhou2010stable} who proved the effectivenss of this model, we estimate $L$ and $S$ via
\begin{equation}
    \min_{L,S}\ \|L\|_\star + \lambda \|S\|_1 
    \quad \mbox{s.t.}\quad 
    \| L + S - M \|_F \leq \varepsilon,
    \label{eq: problem-spcp} 
    \tag{SPCP}
\end{equation} 
where $\lambda = 1/\sqrt{m}$ (as chosen in \cite{candes2011robust}).
Letting $U \Sigma V^\top$ be the SVD for $X$, the proximal operator for the nuclear norm $\|\cdot\|_\star$ is the singular value threshold (see \cite{cai2010singular}):
\begin{equation}
    \mbox{svt}(X,\alpha)
    \triangleq \prox_{\alpha \|\cdot\|_\star}(W) 
    = U \, \mbox{diag}(\max\{\sigma_i - \alpha,\ 0\})\, V^\top.
    \label{eq: svt} 
\end{equation} 
Rather than compute an SVD, the threshold $\mbox{svt}(\cdot,\alpha)$ may also be computed more quickly without an SVD by using a polar decomposition and projection \cite{cai2013fast}.
In the framing of (\ref{eq: problem}), here the matrix $A$ takes the form $A = [\II\ \II]$, where we set $X = [X_L; X_S]$.  This simple structure enables the projection formula in Proposition \ref{prop: projection} to admit a simple, explicit expression, as outlined by the following lemma. (See Appendix \ref{app: spcp} for a proof.)

\def\lemmaProjectionSPCP{
    If $\sC \triangleq \{ X = [X_L; X_S] : \|X_L + X_S - M\|_F \leq \varepsilon\}$, then 
    \begin{equation} 
        \sP_\sC(Z)
        = 
        \left[ 
            \begin{array}{c}
                Z_L - \mu_Z (Z_L + Z_S - M) \\ 
                Z_S - \mu_Z (Z_L + Z_S - M)
            \end{array}
        \right],
    \end{equation}
    where\footnote{We adopt the convention of using $-\varepsilon/0 = -\infty$.}
    \begin{equation}
        \mu_Z 
        \triangleq 
        \max\left\lbrace 
        0,\ \frac{\|Z_L + Z_S - M\|_F - \varepsilon}{2\|Z_L + Z_S - M\|_F} 
        \right\rbrace .
    \end{equation}
}

\begin{lemma}[SPCP Projection]
    \label{lemma: projection-spcp}
    \lemmaProjectionSPCP
\end{lemma}

Using the shrink in (\ref{eq: shrink-1}) and singular value threshold in (\ref{eq: svt}) along with the projection formula in Lemma \ref{lemma: projection-spcp}, a special case of Algorithm \ref{alg: pp} for the problem (\ref{eq: problem-spcp}) is given by Algorithm \ref{alg: pp-spcp} below. 

\begin{remark}[Similarity of PP to Proximal Gradient]  
    When $\varepsilon = 0$ and the constraint in (\ref{eq: problem-spcp}) is moved into the objective as a quadratic penalty, 
    the sequence of estimates generated by proximal gradient for this softly constrained version are identical (for a particular stepsize) to the sequence $\{Z^k\}$ generated by PP.  
    The difference is PP defines solution estimates as the projection of $Z^k$ onto the set of all $X$ satisfying $X_L + X_S= M$.
\end{remark}

\begin{algorithm}[t]
    \caption{{\tt Proximal Projection for Stable Principal Component Pursuit (PP-SPCP)}}
    \label{alg: pp-spcp}
    \setstretch{1.5}
    \def\tab{\hspace*{15pt}}
    \begin{algorithmic}[1]
        \STATE{
            \texttt{PP-SPCP}($M$, $\lambda$, $\varepsilon$):
        }
        \STATE{
            \tab {\bf initialize}  parameter $\alpha > 0$, matrices $Z _L \leftarrow M$,\  $Z_S \leftarrow 0$, $X_L \leftarrow M$, $X_S \leftarrow 0$
        }
        \STATE{
            \tab {\bf while} stopping criteria not met
        }

        \STATE{
            \tab\tab $Z_L \leftarrow Z_L + \mbox{svt}\left( 2X_L - Z_L, \alpha\right) - X_L$
            \hfill $\vartriangleleft$ Use (\ref{eq: svt}) or method in \cite{cai2013fast}
        }
        \STATE{
            \tab\tab $Z_S \leftarrow Z_S + \mbox{shrink}\left( 2X_S -Z_S ,\ \alpha\lambda\right)  - X_S $
            \hfill $\vartriangleleft$ Use (\ref{eq: shrink-1})
        }        
        \STATE{
            \tab\tab $\mu \leftarrow \max\left\lbrace \frac{\|X_L + X_S-M\|_F - \varepsilon}{2 \|X_L + X_S - M\|_F},\ 0 \right\rbrace $
            \label{alg: pp-spcp-mu}
            \hfill $\vartriangleleft$ Treat $-\varepsilon/0$ as $-\infty$
        }
        \STATE{
            \tab\tab $X_L \leftarrow Z_L - \mu (Z_L + Z_S - M)$
        }
        \STATE{
            \tab\tab $X_S \leftarrow Z_S - \mu (Z_L + Z_S - M)$
        }

        \STATE{
            \tab {\bf return} $X_L$ and  $X_S$
        }
    \end{algorithmic} 
    \end{algorithm}

\begin{remark}[Proximal Gradient and SPCP]
   \label{remark: pg-problem}
   Prior work (\textit{e.g.} \cite{zhou2010stable}) uses proximal gradient (PG) for (\ref{eq: problem-spcp}), with the constraint moved into the objective via a quadratic penalty. The weight of this penalty is important. If it is too large, then the constraint is satisfied in (\ref{eq: problem-spcp}) and the objective for PG is suboptimal. If it is too small, then the constraint is not satisfied. For a particular weight, dependent on $M$ and $\varepsilon$, PG is well-known to solve (\ref{eq: problem-spcp}).
    However, we note PG is is generally not apt  when it is unclear how to pick the penalty weight. 
\end{remark}

We compare PP to a variant of alternating splitting augmented Lagrangian method (VASALM) \cite{tao2011recovering} and a ``partially smoothed'' proximal gradient\footnote{In \cite{aybat2014efficient}, FISTA accleration was used; however, as noted above, unaccelerated variants are compared in this work.} (PSPG) on a variation (\ref{eq: problem-spcp-smoothed}) of (\ref{eq: problem-spcp}) wherein the nuclear norm is smoothed (via a tunable parameter $\mu$); this smoothing was proposed in \cite{aybat2014efficient}. 
Note VASALM, PSPG, and PP have essentially identical per-iteration costs (dominated by $\mbox{svt}(\cdot)$).
Here $M$ encodes a sequence of images in a video derived from the PNNL Parking Lot 1 dataset \cite{pnnl2012parking}, with each image vectorized to form a column of $M$ (see Figure \ref{fig: PCP-images}).
Figure \ref{fig: plots-spcp}a shows PP and PSPG\footnote{Constraint violation for PSPG (beyond floating point precision) occurs due to inexactly solving a 1D root-finding problem.} are feasible at each iteration while VASALM approaches feasibility asymptotically.
Due to this feasibility, objective values for PP and PSPG are bounded from below by the optimal value. On the other hand, VASALM oscillates around the optimal value. Figures \ref{fig: plots-spcp}b and \ref{fig: plots-spcp}c shows update residuals and objective values of PP and VASALM have comparable convergence speed, with slight favor toward PP. Meanwhile, PSPG is notably slower. 
Overall, PP  performs at least as well as VASALM and PSPG.
 
\begin{figure}[H]
    \centering  
    \vspace*{-5pt}
    \subfloat[Violation $\max\{ \|X_L^k + X_S^k - M\|_F-\varepsilon,0\} / \varepsilon$]{
        \hspace*{5pt}
        \includegraphics[page=1]{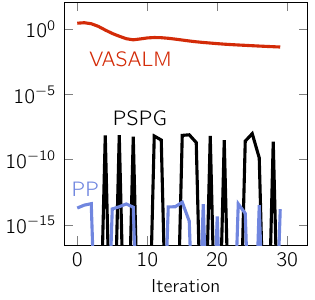} 
        }    
    \subfloat[Objective   $\|L^k\|_\star + \lambda \|S^k\|_1$]{ 
        \includegraphics[page=2]{images/stable-principal-component-pursuit/spcp-plots.pdf}
        \hspace*{5pt} 
        } 
    \subfloat[Residual $\|X^{k+1} - X^k\|_F/\|M\|_F$]{ 
        \includegraphics[page=3]{images/stable-principal-component-pursuit/spcp-plots.pdf} 
        }  

    \caption{Convergence comparison for PP, VASALM and PSPG for (\ref{eq: problem-spcp}). Each PP estimate is feasible to machine precision and PSPG estimates are feasible to accuracy of 1D root-finding solver.
    }
    \label{fig: plots-spcp} 
\end{figure}

\newpage

\begin{algorithm}[H]
    \caption{{\tt Proximal Projection for Computation of Earth Mover's Distance (PP-EMD)}}
    \label{alg: pp-emd}
    \setstretch{1.45}
    \def\tab{\hspace*{15pt}}
    \begin{algorithmic}[1]
        \STATE{
            \texttt{PP-EMD}($\rho^0$, $\rho^1$):
        }
        \STATE{
            \tab {\bf initialize}   matrix $z \in\bbR^{2(n-1)\times n}$ and parameter $\alpha > 0$ 
        }
        \STATE{
            \tab {\bf while} stopping criteria not met
        }

        \STATE{
            \tab\tab {\bf if} $\|\mbox{div}(z) + \rho^1 - \rho^0 \|_F \leq \varepsilon$
        }
        \STATE{
            \tab\tab\tab $m\leftarrow z$
        }        
      
        \STATE{
            \tab\tab {\bf else}
        }
        \STATE{
            \tab\tab\tab $\tau \leftarrow \mbox{solution}\left( 1 = \tau \|Y_\tau\|_F \right)$
            \hfill $\vartriangleleft$ Use root-finding algorithm and $Y_\tau$ in Lemma \ref{lemma: projection-emd}
        }
        \STATE{
            \tab\tab\tab $m \leftarrow z - K^\top U\, [Y_\tau ;\, Y_\tau^\top ]\, U^\top$
            \hfill $\vartriangleleft$ Use $U$ and $Y_\tau$ as in Lemma \ref{lemma: projection-emd}
        }          

        \STATE{
            \tab\tab $z \leftarrow z + \mbox{shrink}(2m - z, \alpha)  - m$
            \hfill $\vartriangleleft$ Use (\ref{eq: shrink-1})  
        }
        \STATE{
            \tab {\bf return} $m$
        }
    \end{algorithmic}
    \end{algorithm}

\subsection{Earth Mover's Distance} \label{subsec: emd}
The earth mover's distance (EMD) is a key metric that is widely used in several fields (\textit{e.g.} image processing \cite{boltz2010earth}, statistics \cite{panaretos2019statistical}, image retrieval \cite{rubner2000earth}, seismology \cite{engquist2014application,metivier2016measuring}, machine learning \cite{tolstikhin2017wasserstein}). This numerical example shows PP can be an efficient method for estimation of EMD, outperforming some existing alternatives. \\

Consider two distributions $\rho^0$ and $\rho^1$, which are here represented as matrices in $\bbR^{n\times n}$. Let $m^1$ and $m^2$ together denote components of a 2D flux. The divergence operator is here defined for a grid size $h > 0$ by\\
\begin{equation}
    \mbox{div}(m) 
    \triangleq \frac{1}{h} \left(m^1_{i,j} - m_{i-1,j}^1 + m_{i,j}^2 - m_{i,j-1}^2\right).
\end{equation}

Following similarly to \cite{li2018parallel}, we estimate the earth mover's distance\footnote{PP immdiately extends to other distances (\textit{e.g.} L2).  To keep experiments concise, we restrict scope to L1.} (\textit{i.e.} Wasserstein-1) between $\rho^0$ and $\rho^1$ as the optimal objective value for the problem \\[-7pt]
\begin{equation} 
    \min_{m} \|m\|_{1}     
    \quad \mbox{s.t.}\quad 
    \|\mbox{div}(m) + \rho^1 - \rho^0 \|_F \leq \varepsilon,
    \label{eq: problem-emd}
    \tag{EMD$_\varepsilon$}
\end{equation}  
where $\varepsilon \geq 0$ is small. (See the note in Remark \ref{remark: emd-eps}  below about $\varepsilon$.)
Neumann boundary conditions are enforced by taking $m_{0,j}^1 = m_{i,0}^2 = m_{n,j}^2 = m_{i,n}^2 = 0$ for all $i,j=1,2,\ldots,n$.  
To write the divergence in matrix form, let $K$ be the backward differencing operator, \textit{i.e.} \\[-4pt]
\begin{equation}
    K \triangleq \frac{1}{h} \left[\begin{array}{rrrrr}
        1 &  \\ 
        -1 & 1 &  \\ 
        &   \ddots & \ddots \\  
        & &   -1 & 1  \\
         & & &    -1 \\  
    \end{array}\right] \in \bbR^{n\times (n-1)}. 
\end{equation}
By eliminating flux entries that are zero due to boundary conditions, we may set $m = [m^1; (m^2)^\top] \in \bbR^{2(n-1)\times n}$.
Thus, application of the divergence and its transpose can be expresed via\footnote{The notation with left multiplication by $A$ and $A^\top$ is somewhat ``abusive'' as right multiplications and transposes are included.}
\begin{equation}
    Am 
    = \mbox{div}(m) 
    = K m^1 + m^2 K^\top \in \bbR^{n\times n}
    \quad\mbox{and}\quad 
    A^\top b 
    = \left[\begin{array}{c} K^\top b \\ (bK)^\top \end{array}\right] 
    = \left[\begin{array}{c} K^\top b \\ K^\top b^\top \end{array}\right] \in \bbR^{2(n-1)\times n}.
    \label{eq: emd-A-operation}
\end{equation}

With this notation established, the projection formula in Proposition \ref{prop: projection} can be rewritten for this setting  as below (see Appendix \ref{app: emd} for a proof).
Importantly, we note the divergence operator here does \textit{not} have full row rank; the fact $\varepsilon > 0$ is what ensures the needed matrix inversion is possible.

\def\lemmaProjectionEMD{
    If $\sC \triangleq \{ m : \|\mbox{div}(m) + \rho^1 - \rho^0 \|_F \leq \varepsilon\}$ and
    $U\Sigma V^\top$ is the SVD for $K$, then 
   \begin{equation}
       \sP_{\sC}(z) 
       = 
       \begin{cases}
       \begin{array}{cl}
           z & \mbox{if}\ \|\mbox{div}(z) + \rho^1 - \rho^0\|_F \leq \varepsilon \\ 
           z - K^\top U  \left[ \begin{array}{c}  Y_\tau   
           \\ Y_\tau^\top   \end{array}\right] U^\top
           & \mbox{otherwise,}
       \end{array} 
       \end{cases}
   \end{equation}
   where $\tau > 0$ satisfies   
   $
       1 = \tau\| (Y_\tau)\|_F 
   $ and $Y_\tau$ is the matrix defined element-wise via 
   \begin{equation}
       (Y_\tau)_{ij}
       \triangleq \frac{[U^\top (Kz+zK^\top + \rho^1 - \rho^0)U]_{ij}}{\sigma_i^2 + \sigma_j^2 + \varepsilon\tau}, 
       \quad\mbox{for all}\ i,j  \in [n].
   \end{equation} 
}

\begin{lemma}[EMD Projection]
    \label{lemma: projection-emd}
    \lemmaProjectionEMD
\end{lemma}

\vspace*{-20pt}

\begin{figure}[H]   
    \centering   
    \subfloat[Standing Cat $\rho^0$]{\includegraphics[height=1.7in]{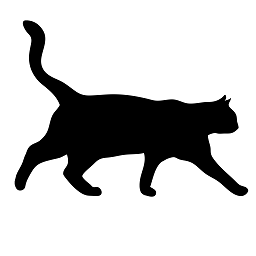}}
    \hspace*{15pt}
    \subfloat[Crouched Cat $\rho^1$]{\includegraphics[height=1.7in]{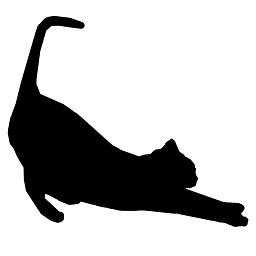}} 
    \hspace*{15pt}
    \subfloat[PP Flux Estimate $m^k$]{\includegraphics[height=1.8in]{\directory/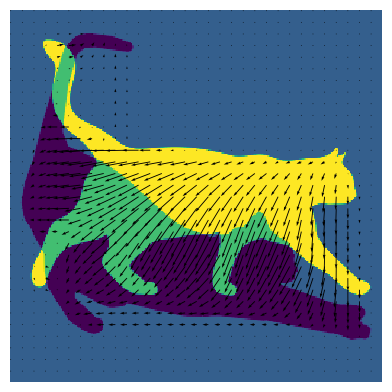}}
    \caption{The black portion of the cats gives the supports $\rho^0$ and $\rho^1$.  Arrows in (c) show the 2D flux $m^k$.}
    \label{fig: cat-images}
\end{figure} 
 
\begin{remark} 
    \label{remark: emd-eps}
    Setting $\varepsilon = 0$ gives the true EMD. 
    For this reason, here we refer to violation as $\|\mbox{div}(m^k) + \rho^1 - \rho^0\|_F$. However, we emphasize using PP with small $\varepsilon > 0$ yields lower violation than using PDHG with $\varepsilon = 0$. Although counterintuitive, this occurs since each iterate of PP is feasible while PDHG obtains feasibility asymptotically.
\end{remark}

We compare PP to PDHG \cite{li2018parallel} and G-Prox PDHG \cite{jacobs2019solving}.
Here, $\varepsilon = 10^{-10}$, and $\rho^0$ and $\rho^1$ are shown by cat images in Figure \ref{fig: cat-images}. (See Appendix \ref{app: emd} for details.) 
The updates for PP and G-Prox PDHG are closely related, which is reflected by significant overlap of their plots in Figure \ref{fig: plots-emd}.
Figure \ref{fig: plots-emd}a shows  PP has violation no more than $ \varepsilon$ for each iteration. Although the same $\varepsilon$ is used for projections onto divergence free vector fields in G-Prox PDHG, its violation grows as that scheme accumulates projection errors. Hence PP may better handle constraints than G-Prox PDHG. The times to run PP, G-Prox PDHG, and PDHG for 20,000 iterations were, respectively \emdTimePP s, \emdTimePDHG s, and \emdTimeGPDHG s (averaged over 10 trials). The per-iteration cost of PP   roughly equals that of G-Prox PDHG and \emdTimePPoverPDHG  X that of PDHG. Despite these higher per-iteration costs, PP and G-Prox PDHG are  more efficient than PDHG here as PDHG requires order(s) of magnitude more iterations to converge.

\begin{figure}[H] 
    \centering  
    \subfloat[Violation $\| \mbox{div}(m^k) + \rho^1 - \rho^0\|_F$]{\includegraphics[page=1]{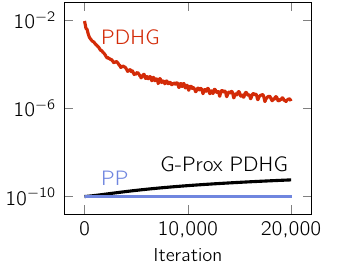}}    
    \hspace*{5pt}
    \subfloat[Objective   $\|m^k\|_{1}$]{\includegraphics[page=2]{images/earth-movers-distance/emd-plots.pdf}}
    \hspace*{5pt}
    \subfloat[Residual $\|m^{k+1} - m^k\|_F$]{\includegraphics[page=3]{images/earth-movers-distance/emd-plots.pdf}}
    \caption{Convergence comparison for PP, PDHG, and G-Prox PDHG for computation of earth mover's distance between cat images. PP and G-Prox PDHG have significantly overlapping plots in (b) and (c).}
    \label{fig: plots-emd}  
\end{figure}

\newpage
 
\subsection{Stable Matrix Completion} \label{subsec: mc}

A key task in many applications consists of filling in missing entries of a partially observed matrix. Forms of this problem arise in machine learning, system identification, recommendations systems, localization in IoT networks, image restoration, and more (\textit{e.g.} see \cite{ramlatchan2018survey,davenport2016overview,nguyen2019low} and the references therein). 
Although infinitely many matrices are consistent with partial observations (causing recovery to be ill-posed), a common approach is to find the matrix with minimal nuclear norm that is consistent with the observations.  
Indeed, the seminal work \cite{candes2012exact} shows many low-rank matrices $M \in \bbR^{n_1\times n_2}$ can be \textit{exactly} recovered by solving the convex problem
\begin{equation}
    \min_{X\in\bbR^{n_1\times n_2}} \|X\|_\star 
    \quad\mbox{s.t.}\quad 
    X_{ij} = M_{ij}
    \ \mbox{for all}\ (i,j)\in \Omega,
    \label{eq: problem-mc}
    \tag{MC}
\end{equation} 
where $\Omega \subseteq [n_1]\times [n_2]$ is the set of   observed indices. For settings with observations corrupted by Gaussian noise,  we consider the stable matrix completion variant\footnote{Note (\ref{eq: problem-smc}) and (\ref{eq: problem-spcp}) can both be expressed as special cases of a common, more general formulation.} proposed by \cite{candes2010matrix}:
\begin{equation}
    \min_{X\in\bbR^{n_1\times n_2}} \|X\|_\star 
    \quad\mbox{s.t.}\quad 
    \|\sP_\Omega(X-M) \|_F \leq \varepsilon.
    \label{eq: problem-smc}  
    \tag{SMC}  
\end{equation}
where we let $\sP_\Omega$ denote the projection onto the the span of matrices vanishing outside of $\Omega$ and $\sP_{\Omega^\perp}$ be the projection onto the complement of this space, \textit{i.e.}
\begin{equation}
    \sP_{\Omega}(X)_{ij}
    \triangleq \begin{cases}
        \begin{array}{cl}
            X_{ij} & \mbox{if}\ (i,j)\in \Omega \\ 
            0 & \mbox{otherwise}
        \end{array}
    \end{cases}
    \quad\mbox{and}\quad\quad
    \sP_{\Omega^\perp}(X)_{ij}
    \triangleq \begin{cases}
        \begin{array}{cl}
            X_{ij} & \mbox{if}\ (i,j)\notin \Omega \\ 
            0 & \mbox{otherwise.}
        \end{array}
    \end{cases}    
    \label{eq: proj-mc}
\end{equation}

The projection formula from Proposition \ref{prop: projection} simplifies in this setting as follows (see Appendix \ref{app: smc} for a proof).

\def\lemmaProjectionSMC{
    If $\sC \triangleq \{  X : \| \sP_\Omega(X-M) \|_F \leq \varepsilon \}$, then 
    \begin{equation}
        \sP_\sC(Z)  
        = \begin{cases}
            \begin{array}{cl}
                Z & \mbox{if}\ \| \sP_\Omega(Z-M) \|_F \leq \varepsilon  \\ 
                \sP_\Omega\left( \frac{[ \|\sP_\Omega(Z-M)\|_F - \varepsilon]M + \varepsilon Z}{\|\sP_\Omega(Z-M)\|_F}\right) + \sP_{\Omega^\perp}(Z) & \mbox{otherwise,}
            \end{array}
        \end{cases}
    \end{equation}
}

\begin{lemma}[SMC Projection]
    \label{lemma: projection-mc}
    \lemmaProjectionSMC
\end{lemma}

With this notation and projection formula established,  the special case of PP for (\ref{eq: problem-smc}) is   Algorithm \ref{alg: pp-mc} below.

\begin{algorithm}[H]
    \caption{{\tt Proximal Projection for Stable Matrix Completion (PP-SMC)}}
    \label{alg: pp-mc}
    \setstretch{1.35}
    \def\tab{\hspace*{15pt}}
    \begin{algorithmic}[1]
        \STATE{
            \texttt{PP-SMC}($\sP_\Omega(M), \ \varepsilon$):
        } 
        \STATE{
            \tab {\bf initialize} parameter $\alpha > 0$,  matrices $Z \leftarrow \sP_\Omega(M)$, $X\leftarrow\sP_\Omega(M)$
            \hfill $\vartriangleleft$ Initialize via observed entries
        }

        \STATE{
            \tab {\bf while} stopping criteria not met
        }   
        \STATE{
            \tab\tab $Z \leftarrow Z + \mbox{svt}(2X - Z,\ \alpha) - X$
            \hfill $\vartriangleleft$ Use (\ref{eq: svt})
        }        
        \STATE{
            \tab\tab {\bf if} \ $\|\sP_\Omega(Z-M)\|_F \leq \varepsilon$
            \hfill $\vartriangleleft$ Use (\ref{eq: proj-mc})
        }  
        \STATE{
            \tab\tab\tab $X \leftarrow  Z$
        }          
        \STATE{
            \tab\tab {\bf else}  
        }  
        \STATE{
            \tab\tab\tab $X \leftarrow  \sP_\Omega\left( \frac{[ \|\sP_\Omega(Z-M)\|_F - \varepsilon]M + \varepsilon Z}{\|\sP_\Omega(Z-M)\|_F}\right) + \sP_{\Omega^\perp}(Z) $
            \hfill $\vartriangleleft$ Use (\ref{eq: proj-mc})
        }

        \STATE{
            \tab {\bf return} $X$
        }
    \end{algorithmic}
    \end{algorithm}  

This numerical example consists of solving (\ref{eq: problem-smc}) for various Gaussian matrices $M$. Letting $n=1,000$, we generated $n\times n$ matrices $M$ of rank $r$ by sampling two $n\times r$ factors $M_L$ and $M_R$ independently, each having i.i.i.d Gaussain entries and then setting $M = M_L M_R^\top$.
The set of observed entries $\Omega$ is uniformly sampled among sets of cardinality $s$. The observations of $M$ are corrupted by noise $N$, which is also Gaussian, and $\varepsilon$ is set to be the Frobenius norm of the noise in the observations. Results are shown using various choices of $s$ and $r$, and we note the degrees of freedom $d_r$ of an $n\times n$ matrix of rank $r$ is $d_r = r\cdot (2n-r)$.

\newpage

\begin{figure}[t]
    \centering
    \subfloat[Violation $\dfrac{\max\{\|\sP_\Omega(X^k - M)\|_F - \varepsilon,0\}}{\varepsilon}$ ]{
        \includegraphics[page=1]{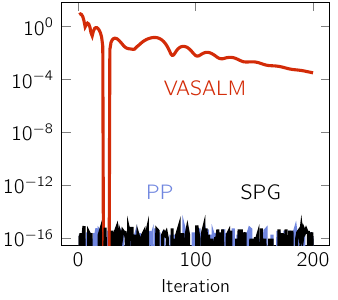}
    } 
    \subfloat[  Objective $\|X^k\|_\star$] {
        \hspace*{5pt}
        \includegraphics[page=2]{images/matrix-completion/smc-plots.pdf}
        \hspace*{5pt}
    }
    \subfloat[Residual $\dfrac{\|X^{k+1} - X^k\|_F}{\|M\|_F}$]{
        \includegraphics[page=3]{images/matrix-completion/smc-plots.pdf}
    }    
    \caption{Comparison of PP, VASALM, and SPG for solving (\ref{eq: problem-smc}) with $\mbox{rank}(M) = r = 10$ and $|\Omega| = s = 5d_r$.  
    Plots show medians from 10 trials with distinct random seeds.
    }
    \label{fig: plots-smc-10}
\end{figure}     
 
\begin{figure}[H] 
    \centering
    \subfloat[Violation $\dfrac{\max\{\|\sP_\Omega(X^k - M)\|_F - \varepsilon,0\}}{\varepsilon}$ ]{
        \includegraphics[page=4]{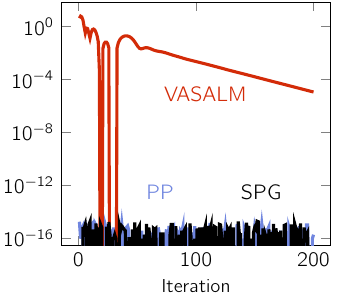}
    } 
    \subfloat[  Objective $\|X^k\|_\star$] {
        \hspace*{5pt}
        \includegraphics[page=5]{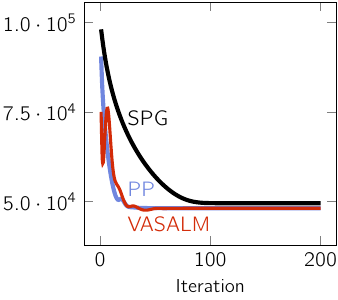}
        \hspace*{5pt}
    }
    \subfloat[Residual $\dfrac{\|X^{k+1} - X^k\|_F}{\|M\|_F}$]{
        \includegraphics[page=6]{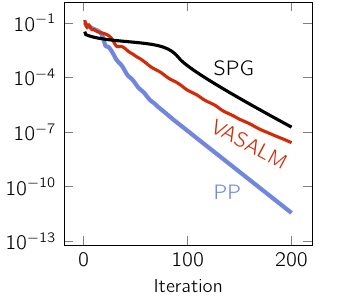}
    }    
    \caption{Comparison of PP, VASALM, and SPG for solving (\ref{eq: problem-smc}) with $\mbox{rank}(M) = r = 50$ and $|\Omega|=s = 4d_r$.   
    Plots show medians from 10 trials with distinct random seeds. 
    }
    \label{fig: plots-smc-50}
\end{figure}   

In similar fashion to the example for (\ref{eq: problem-spcp}), here we compare PP to (VASALM) \cite{tao2011recovering} and a ``smoothed'' proximal gradient (SPG) \cite{aybat2014efficient}  on a variation (\ref{eq: problem-smc-smoothed}) of (\ref{eq: problem-smc}) wherein the nuclear norm is smoothed (via a tunable parameter $\mu$). 
Note PP, VASALM, and SPG all have essentially identical per-iteration costs, which are dominated by the evaluation of $\mbox{svt}(\cdot)$.
The parameters for each algorithm were chosen to give good performance for the case where $r=50$ and $s/d_r = 4$. To demonstrate the flexibility of these algorithms, these parameters were held fixed across all trials and values of $r$ and $s/d_r$. \\

Numerical results are plotted in Figures \ref{fig: plots-smc-10} and \ref{fig: plots-smc-50} and also shared in Table \ref{tab: smc}.
As shown in Figures \ref{fig: plots-smc-10}a and \ref{fig: plots-smc-50}a, 
PP and SPG are feasible at each iteration while VASALM iteratively reduces its constraint violation.    
Figures \ref{fig: plots-smc-10}b and \ref{fig: plots-smc-50}b show the objective $\|X^k\|_\star$ of PP and VASALM converge quickly, and that SPG converged to a suboptimal limit in each case. Figures \ref{fig: plots-smc-10}c and \ref{fig: plots-smc-50}c show PP converges faster than VASALM and both of these converge faster than SPG. 
Similar results are reported in Table \ref{tab: smc}.
The following statements hold for all three configurations shown in the table.  PP converges in fewer steps than VASALM and SPG (\textit{i.e.} about 55\% as many steps as VASALM and 40\% as many as SPG). 
The objective values $\|X^k\|_\star$ of PP and VASALM agree to three digits of accuracy.  The output of VASALM has a violation exceeding machine precision. SPG performs worse than PP with respect to every metric. Overall, PP outperforms VASALM and SPG in these examples.


\begin{table}[t] 
    \def\smcItersPPa{105}
\def\smcViolPPa{1.80e-16}
\def\smcObjPPa{9.49e+03}
\def\smcItersPPb{61}
\def\smcViolPPb{3.28e-16}
\def\smcObjPPb{4.82e+04}
\def\smcItersPPc{54}
\def\smcViolPPc{1.16e-15}
\def\smcObjPPc{9.56e+04}
\def\smcItersVASALMa{185}
\def\smcViolVASALMa{5.45e-04}
\def\smcObjVASALMa{9.49e+03}
\def\smcItersVASALMb{110}
\def\smcViolVASALMb{1.35e-03}
\def\smcObjVASALMb{4.82e+04}
\def\smcItersVASALMc{96}
\def\smcViolVASALMc{1.30e-03}
\def\smcObjVASALMc{9.56e+04}
\def\smcItersSPGa{247}
\def\smcViolSPGa{3.24e-16}
\def\smcObjSPGa{1.09e+04}
\def\smcItersSPGb{145}
\def\smcViolSPGb{9.29e-16}
\def\smcObjSPGb{4.96e+04}
\def\smcItersSPGc{133}
\def\smcViolSPGc{6.02e-16}
\def\smcObjSPGc{9.69e+04}

    \small  
    \centering 
    \setstretch{1.3}
    \begin{tabular}{ccc|ccc|ccc|ccc}
        \multicolumn{3}{c|}{Unknown Matrix $M$}
        & \multicolumn{9}{c}{Result} \\\hline  
        \multirow{2}{*}{rank $(r)$}
        & \multirow{2}{*}{$s/d_r$}
        & \multirow{2}{*}{$s/n^2$} 
        &
        \multicolumn{3}{c|}{PP}
        & \multicolumn{3}{c|}{VASALM} 
        & \multicolumn{3}{c}{SPG} \\
         & & & \# & Viol & $\|X^k\|_\star$ & \# & Viol & $\|X^k\|_\star$ & \# & Viol  & $\|X^k\|_\star$  \\\hline
         10  & 5 & 0.0995 & 
         {\bf \smcItersPPa} & \smcViolPPa   & {\smcObjPPa} & 
         \smcItersVASALMa & \smcViolVASALMa   & \smcObjVASALMa  & 
         \smcItersSPGa & \smcViolSPGa & \smcObjSPGa \\
         50  & 4 & 0.3900 & 
         {\bf \smcItersPPb} & \smcViolPPb & { \smcObjPPb} & 
         \smcItersVASALMb&  {\smcViolVASALMb} & {\smcObjVASALMb} &
         \smcItersSPGb & \smcViolSPGb & \smcObjSPGb\\
         100 & 3 & 0.5700 & 
        {\bf \smcItersPPc} & \smcViolPPc & { \smcObjPPc} &  
         \smcItersVASALMc &  {\smcViolVASALMc} &  \smcObjVASALMc &
         \smcItersSPGc & \smcViolSPGc &\smcObjSPGc \\  
    \end{tabular}
    \caption{Results for stable matrix completion example with $M \in \bbR^{n\times n}$ and $n=1,000$. 
    Each algorithm executed until $\|X^{k+1} - X^k \|_F / \|M\|_F \leq 10^{-5}$. The number of iterations $k$ to meet this stopping condition is denoted by \#.  The relative constraint violation $\max\{ \|\sP_\Omega(X^k-M)\|_F - \varepsilon,0\}/\varepsilon$ is denoted by ``Viol.''
    Note an $n \times n$ matrix of rank $r$ depends upon $r(2n - r)$ degrees of freedom $d_r$. Results shown are averages for 10 random seeds.
    }
    \label{tab: smc}
\end{table}

\section{Discussion}  

This work proposes the proximal projection (PP) algorithm for minimizing proximable functions subject to satisfication of linear constraints within a specified tolerance. It has a simple and easy-to-code formulation. The primary novelty herein is showing how to project onto the constraint set. This can be efficiently computed by solving a 1D root-finding problem; in some important instances, explicit formulas are available. 
This projection is interwoven with proximal operations via Douglas-Rachford splitting to obtain PP.
Our formal analysis shows PP converges, with any choice of step size, to an optimal solution. Moreover, unlike many algorithms that only achieve feasibility asymptotically, the output of PP is feasible  for any stopping iteration. \\
 
 Numerically, we show PP performs favorably against alternatives in a few important applications: basis pursuit, stable principal component pursuit, earth mover's distances, and matrix completion. In each   example provided, PP is numerically verified to be feasible and to converge to an optimal solution.
 Moreover, PP performs at least as well as each alternative method in these examples with respect to every metric considered.
 \\
 
The presented projection formula provided can be included in other algorithms (\textit{e.g.} projected gradient when the objective is smooth); we leave investigation of this to future work.  Acceleration techniques \cite{themelis2019supermann,zhang2020globally,park2022exact} can be used on Douglas-Rachford splitting, which therefore apply to PP; future work could examine the performance of these. Lastly, extensions may also involve PP's incorporation as an optimization layer in machine learning applications (\textit{e.g.} the realm of learning-to-optimize and predict-then-optimize \cite{chen2021learning,amos2017optnet,shlezinger2023model,mckenzie2024differentiating,elmachtoub2022smart}).

\section*{Acknowledgements} 
We thank Samy Wu Fung for his helpful discussions and feedback on early drafts of this manuscript.

{\small 
    \bibliographystyle{unsrt} 
    \bibliography{refs} 
}

\newpage 
\appendix 
\section{Proofs of Main Results} \label{app: proofs}

\def\sL{\mathcal{L}} 

Each of our results from Section \ref{sec: results} is proven below. For ease of reference,   results are restated before their proof. We begin with an auxiliary lemma used for the proof of Lemma \ref{prop: projection}. \\ 

\begin{lemma}[Unique Positive $\tau_x$] \label{lemma: tau-exists}
    \label{lemma: tau} If \ref{c: full-rank} and \ref{c: b-image} hold and $x\in\bbR^n$ is such that $\|Ax-b\| > \varepsilon$, then there is   unique $\tau_x \in\bbR$ satisfying 
    \begin{equation}
        0 < \tau_x \leq \frac{\sigma_{max}(A)^2}{\|Ax-b\|-\varepsilon}
        \label{eq: tau-bounds}
    \end{equation}
    for which 
    \begin{equation}
        \tau_x \| (AA^\top + \varepsilon\tau_x)^{-1} Ax\| = 1.
        \label{eq: lemma-tau-x-condition}
    \end{equation}
\end{lemma}
\begin{proof}   
    We proceed in the following manner. First, we obtain an alternative expression for the left hand side of (\ref{eq: lemma-tau-x-condition}) in terms of the SVD of $A$, which is continuous in $\tau$ (Step 1). For the $\varepsilon = 0$ case, we then establish uniqueness and the bounds for  $\tau_x$   in (\ref{eq: tau-bounds}) (Step 2). For the $\varepsilon > 0$ case, we establish existence of $\tau_x$ in the desired interval by  application of the intermediate value theorem  (Step 3). Lastly, uniqueness of $\tau_x$ is   established for the $\varepsilon > 0$ case (Step 4). For notational compactness, we henceforth set 
    \begin{equation}
        \phi(\tau) \triangleq \|\tau (AA^\top + \varepsilon\tau\mbox{I})^{-1}(Ax-b)\|^2.
    \end{equation} 

    \textbf{Step 1}. Let $U\Sigma V^\top$ be the   singular value decomposition of $A$ with $\Sigma = \mbox{diag}(\sigma_i)$. Also let \\ $\hat{\sigma}  = (\sigma_1,\sigma_2,\ldots,\sigma_r,0,\ldots,0)$ be the vector of $r \leq \min\{m,n\}$ singular values $\sigma_i$ of $A$ in descending order, padded with zeros as needed to yield $\hat{\sigma}\in \bbR^m$.
    Then 
    \begin{equation}
        AA^\top + \varepsilon \tau \II 
        = U\Sigma V^\top V \Sigma^\top U^\top  + \varepsilon\tau\II 
        = U\, \mbox{diag}(\hat{\sigma}_i^2 + \varepsilon\tau) U^\top,
        \quad\mbox{for all}\ \tau > 0,
    \end{equation}
    where the final equality holds since  $U$ and $V$ are orthogonal.
    Whenever $\tau > 0$,   \ref{c: full-rank} ensures $\hat{\sigma}_i^2 + \varepsilon \tau > 0$ for each index $i\in[m]$, and so this matrix is invertible. In particular, using the orthogonality of $U$,
    \begin{equation}
        \tau  (AA^\top + \varepsilon \tau\II)^{-1} 
        = U\, \mbox{diag}\left( \frac{\tau}{\hat{\sigma}_i^2 + \varepsilon\tau}\right)U^\top  
        \quad\mbox{for all}\ \tau > 0,
        \label{eq: proof-tau-inverse-formula}
    \end{equation}
    and so, again using the orthogonality of $U$,
    \begin{equation}
       \phi(\tau)
        = \| \tau (AA^\top + \varepsilon \tau \II)^{-1} (Ax-b)\|^2
        = \left\|   \mbox{diag}\left( \frac{\tau}{\hat{\sigma}_i^2 + \varepsilon\tau}\right)U^\top (Ax-b)\right\|^2,
        \quad\mbox{for all}\ \tau > 0.
    \end{equation}   
    Note each fraction $  \tau / (\hat{\sigma}_i^2 + \varepsilon\tau)$ is continuous on $(0,\infty)$, and the entire expression, being a composition of continuous functions, is also continuous. \\

    \textbf{Step 2}.
    If $\varepsilon = 0$, then      $A$ has full row-rank by \ref{c: full-rank}, which implies $AA^\top$ is invertible. 
    Additionally, this also yields $(AA^\top)^{-1}(Ax-b) \neq 0$ since $(Ax-b)\neq 0$.
    Thus, letting $\sigma_+$ denote the maximum singular value of $A$,
    \begin{equation}
        0 < \tau_x = \frac{1}{\|(AA^\top)^{-1} (Ax-b)\|}
        = \frac{1}{\left\| U \, \mbox{diag}\left( \hat{\sigma}_i^{-2}\right)U^\top (Ax-b) \right\|}
        \leq \frac{1}{  \sigma_+^{-2} \| U U^\top (Ax-b)\|}
        = \frac{\sigma_{+}^2}{\|Ax-b\|}.
    \end{equation}    
    
    \textbf{Step 3}.
    For the remainder of the proof, we assume $\varepsilon > 0$. 
    To apply the intermediate value theorem, we identify $\tau_-$ for which $\phi(\tau_-) < 1$ and $\tau_+$ for which $\phi(\tau_+) \geq 1$. Indeed, set 
    \begin{equation}
        \tau_+
        \triangleq \frac{\sigma_+^2}{\|Ax-b\|-\varepsilon}
    \end{equation}
    so that 
    \begin{equation}
        \frac{\tau_+}{\hat{\sigma_i}^2 + \varepsilon\tau_+} 
        \geq \frac{\tau_+}{\sigma_+^2 + \varepsilon\tau_+} 
        = \frac{\sigma_+^2}{(\|Ax-b\|-\varepsilon) \sigma_+^2 + \varepsilon \sigma_+^2}
        = \frac{1}{\|Ax-b\|},
        \quad\mbox{for all indices $i\in[m]$.}
    \end{equation}
    This implies
    \begin{align}
        \phi(\tau_+)
        & = \left\|  \mbox{diag}\left( \frac{\tau_+}{\hat{\sigma}_i^2 + \varepsilon\tau_+}\right) U^\top (Ax-b)\right\|^2\\
        & \geq   \left\|  \mbox{diag}\left( \frac{1}{\|Ax-b\|}\right) U^\top (Ax-b)\right\|^2\\
        & = \left(\frac{1}{\|Ax-b\|} \cdot \| \mbox{I} U^\top (Ax-b)\|\right)^2 \\
        & = 1,
    \end{align}  
    \textit{i.e.} $\phi(\tau_+) \geq 1$.
    By \ref{c: b-image}, there is $z$ such that $\|Az - b\| < \varepsilon$, which implies   there is $\xi$ with $\|\xi\| < \varepsilon$ such that 
    \begin{equation}
        b = Az + \xi
        \quad\implies\quad
        Ax - b = A(x-z) - \xi.
    \end{equation}    
    Then note 
    \begin{equation}
        \lim_{\tau\rightarrow 0^+}\| \tau (AA^\top + \varepsilon \tau \II)^{-1} A(x-z)\| = \lim_{\tau\rightarrow 0^+} \left\|  \mbox{diag}\left( \frac{\hat{\sigma}_i\tau}{\hat{\sigma}_i^2 + \varepsilon\tau}\right)V^\top (x-z)\right\|
        = 0,
        \label{eq: proof-limit-tau-0}
    \end{equation}
    where the first equality follows from using the SVD of $A$ and orthogonality of $U$.
    The second equality can be deduced by considering two cases for each index $i$.   First, if $\hat{\sigma}_i = 0$ for any index $i$, then $\hat{\sigma}_i \tau / (\hat{\sigma}_i^2 + \varepsilon \tau) = 0$ for all $\tau > 0$. On the other hand, if $\hat{\sigma}_i > 0$, then the numerator of $\hat{\sigma}_i \tau / (\hat{\sigma}_i^2 + \varepsilon \tau)$ converges to zero while the denominator converges to $\hat{\sigma}_i^2 > 0$.
    Consequently, there is $\tau_- > 0$ such that 
    \begin{equation}
        \| \tau_- (AA^\top + \varepsilon \tau_- \II)^{-1} A(x-z)\|
        = \left\|  \mbox{diag}\left( \frac{\hat{\sigma}_i\tau}{\hat{\sigma}_i^2 + \varepsilon\tau}\right)V^\top (x-z)\right\|
        \leq \frac{\varepsilon-\|\xi\|}{2\varepsilon},
    \end{equation} 
    where we recall $\|\xi\| < \varepsilon$.
    Thus,   
    \begin{align}
        \sqrt{\phi(\tau_-)}
        & = \| \tau_- (AA^\top + \varepsilon \tau_- \II)^{-1} (Ax-b)\|\\
        & = \|\tau_- (AA^\top + \varepsilon \tau_- \II)^{-1}  (A(x-z) -\xi) \| \\
        & = \left\| \mbox{diag}\left(\frac{\tau \hat{\sigma}_i}{\hat{\sigma}_i^2 + \varepsilon\tau}\right) V^\top (x-z) - \mbox{diag}\left(\frac{\tau}{\hat{\sigma}_i^2 + \varepsilon\tau} \right)U^\top \xi \right\| \\
        & \leq \left\| \mbox{diag}\left(\frac{\tau \hat{\sigma}_i}{\hat{\sigma}_i^2 + \varepsilon\tau}\right) V^\top (x-z) \right\| +\left\| \mbox{diag}\left(\frac{\tau}{\hat{\sigma}_i^2 + \varepsilon\tau} \right)U^\top \xi \right\| \\
        & \leq \frac{\varepsilon-\|\xi\|}{2\varepsilon} + \frac{\tau}{0^2 + \varepsilon\tau} \|U^\top \xi\|\\
        & = \frac{\varepsilon-\|\xi\|}{2\varepsilon} + \frac{\|\xi\|}{\varepsilon}\\
        & = \frac{\varepsilon+\|\xi\|}{2\varepsilon} \\
        & < 1. 
    \end{align}
    Hence  $\phi(\tau_-) < 1$.
    Therefore, by the intermediate value theorem, there is $\tau_x \in [\tau_-, \tau_+] \subset (0, \tau_+]$ such that (\ref{eq: lemma-tau-x-condition}) holds, from which (\ref{eq: tau-bounds}) also follows.  \\

    \textbf{Step 4}. All that remains is to verify $\tau_x$ is unique. To show this, it suffices to verify $\phi$ is strictly increasing. 
    Note 
    \begin{equation}
        \frac{\mbox{d}}{\mbox{d}\tau} \left( \frac{  \tau}{\hat{\sigma}_i^2 + \varepsilon \tau }\right)^2
        = \frac{2\hat{\sigma}_i^2 \tau}{(\hat{\sigma}_i^2 + \varepsilon\tau)^3} \geq 0,
        \quad \mbox{for all $i$ and} \ \tau > 0.
    \end{equation}
    Using this fact, differentiating reveals 
    \begin{subequations}
        \begin{align}
            \phi'(\tau)
            & = \frac{\mbox{d}}{\mbox{d}\tau} \left\| \mbox{diag}\left(\frac{\tau}{\hat{\sigma}_i^2 + \varepsilon\tau}\right) U^\top(  Ax-b) \right\|^2\\ 
            & = \frac{\mbox{d}}{\mbox{d}\tau} \sum_{i=1}^m  \left(\frac{\tau}{\hat{\sigma}_i^2 + \varepsilon\tau} \cdot (U^\top (Ax-b))_i\right)^2 \\
            & = \sum_{i=1}^m  \frac{2\hat{\sigma}_i^2\tau}{(\hat{\sigma}_i^2 + \varepsilon\tau)^3}  \cdot (U^\top(Ax-b))_i^2\\
            & \geq  \frac{2}{\varepsilon^3\tau^2}\cdot\sum_{i=1}^m \hat{\sigma}_i^2 (  U^\top(Ax-b))_i^2\\
            & =  \frac{2}{\varepsilon^3\tau^2}\cdot \|   \hat{\Sigma} U^\top (Ax-b)\|^2\\
            & =  \frac{2}{\varepsilon^3\tau^2}\cdot \|   {\Sigma}^\top U^\top (Ax-b)\|^2\\
            & =  \frac{2}{\varepsilon^3\tau^2}\cdot \| V \Sigma^\top U^\top (Ax-b)\|^2\\
            & =  \frac{2}{\varepsilon^3\tau^2}\cdot \| A^\top (Ax-b)\|^2,
            \quad\mbox{for all}\ \tau > 0.
            \label{eq: proof-tau-derivative-norm-end}
        \end{align}\label{eq: proof-tau-derivative-norm}\end{subequations}
        Due to convexity of the least squares function $\|Au-b\|^2$, any point $p$ satisfying the first-order optimality condition 
        \begin{equation}
            0 = [\nabla \|Au-b\|^2]_{u=p} = A^\top (Ap-b)
        \end{equation}
        is a minimizer of $\|Au-b\|^2$, \textit{i.e.}
        \begin{equation}
            A^\top (Ap-b) = 0 
            \quad\implies\quad
            p \in \argmin_{u} \|Au-b\|^2
            \quad\implies\quad
            \|Ap - b\|^2 \leq \|Az-b\|^2 < \varepsilon.
        \end{equation} 
        Since  $\|Ax-b\|^2 > \varepsilon$, it follows that $A^\top (Ax-b) \neq 0$. Hence 
        \begin{equation}
            \phi'(\tau)
            \geq \frac{2}{\varepsilon^3\tau^2}\cdot \| A^\top (Ax-b)\|^2
            > 0,\quad\mbox{for all}\ \tau > 0,
        \end{equation}
        \textit{i.e.} $\phi$ is strictly increasing.  
\end{proof}

\vspace*{0.5in}

The next lemma's proof draws heavily from \cite[Lemma 6.68]{beck2017first}.\\

{\bf Proposition \ref{prop: projection}} (Projection Formula).
\textit{\projectionProposition}
 
\begin{proof}
    The set $\sC$ is nonempty by \ref{c: b-image}. As $\sC$ is also closed and convex, projections onto $\sC$ exist and are unqiue \cite[Theorem 1.2.3]{cegielski2012iterative}.
    Let $u^\star$ be the projection of $x$ onto $\sC$ so that $u^\star$ is the unique solution to 
    \begin{equation}
        \min_{u} \left\lbrace \frac{1}{2}\|u-x\|^2 \ : \ \|Au -b \| \leq \varepsilon\right\rbrace.
        \label{eq: proof-projection-problem}
    \end{equation}
    If $\|Ax-b\|\leq\varepsilon$, then the solution to (\ref{eq: proof-projection-problem}) is $u^\star = x$ since, in that case, the objective value is zero and norms are nonnegative.  Henceforth, we assume $\|Ax-b\| > \varepsilon$.
    Introducing $z \in \bbR^m$,  (\ref{eq: proof-projection-problem}) may be rewritten as 
    \begin{equation}
        \min_{u,z} \left\lbrace \frac{1}{2} \|u-x\|^2 \ : \ z = Au-b, \ \|z\|\leq\varepsilon\right\rbrace .
    \end{equation}
    The Lagrangian $\sL\colon \bbR^n \times \bbR^m \times \bbR_{\geq 0} \times \bbR^m\rightarrow \bbR$ for this problem is 
    \begin{subequations}
    \begin{align}
        \sL(u,z ; \alpha, y) 
        & = \frac{1}{2}\|u-x\|^2 + y^\top (z - [Au- b]) + \alpha ( \|z\| - \varepsilon) \\
        & = \left[ \frac{1}{2}\|u-x\|^2 - (A^\top y)^\top u\right] + \left[ \alpha (\|z\|-\varepsilon) + y^\top (z+b)\right].
    \end{align}
    \end{subequations} 
    Since the Lagrangian $\sL$ is separable with respect to $u$ and $z$, the dual objective can be written as 
    \begin{equation}
        \min_{u,z} \sL(u,z; \alpha,y)
        = \min_u \left[  \frac{1}{2}\|u-x\|^2 - (A^\top y)^\top u\right] 
        + \min_z \left[  \alpha  (\|z\|-\varepsilon) + y^\top (z+b)\right].
        \label{eq: proof-proj-dual-obj-lagrangian}    
    \end{equation}
    The minimizer of the minimization problem for $u$ is $u^\star = x + A^\top y$, which yields 
    \begin{equation}
        \min_u \left[  \frac{1}{2}\|u-x\|^2 - (A^\top y)^\top u\right] 
        = \frac{1}{2}\|u^\star - x\|^2 - (A^\top y)^\top u^\star
        = - \frac{1}{2} \|Ay\|^2 - (Ax)^\top y.
        \label{eq: proof-proj-dual-obj-u}
    \end{equation}
    Due to the linear term $y^\top z$, the minimization of $\sL$ with respect to $z$ is obtained when $z$ is anti-parallel to $y$, for which the expression in the $z$ minimization of (\ref{eq: proof-proj-dual-obj-lagrangian}) becomes 
    \begin{equation}
        \left( \alpha - \|y\| \right) \|z\| - \alpha \varepsilon + y^\top b.
        \label{eq: proof-proj-dual-obj-z-expr}
    \end{equation}
    Consequently, if $\|y\| > \alpha$, the expression in (\ref{eq: proof-proj-dual-obj-z-expr}) goes to $-\infty$ as $\|z\|\rightarrow \infty$. Otherwise, it is minimized by picking $z = 0$. In summary,
    \begin{equation}
        \min_z \left[  \alpha  (\|z\|-\varepsilon) + y^\top (z + b)\right]
        = \begin{cases}
        \begin{array}{cl}
            - \alpha \varepsilon  + y^\top b & \mbox{if}\ \|y\|\leq \alpha \\ 
            - \infty & \mbox{otherwise}.
        \end{array}
        \end{cases}
        \label{eq: proof-proj-dual-obj-z}
    \end{equation}
    Combining (\ref{eq: proof-proj-dual-obj-lagrangian}), (\ref{eq: proof-proj-dual-obj-u}), and (\ref{eq: proof-proj-dual-obj-z}), we obtain the dual problem 
    \begin{equation}
        \max_{\alpha\geq 0,\ y \in\bbR^m} \left\lbrace -\frac{1}{2}\|Ay\|^2 - (Ax)^\top y - \alpha \varepsilon + y^\top b \  : \ \|y\|\leq \alpha \right\rbrace. 
        \label{eq: proof-proj-dual-problem}
    \end{equation}
    Importantly, strong duality holds for the primal-dual pair of problems (\ref{eq: proof-projection-problem}) and (\ref{eq: proof-proj-dual-problem}) \cite[Proposition 6.4.4]{bertsekas2003convex}. Thus, upon finding an optimal dual solution $y^\star$, the primal solution is $u^\star = x + A^\top y^\star$. \\

    The dual objective is strictly decreasing as $\alpha$ increases beyond $\|y\|$, and so the optimal choice for $\alpha$ is always $\alpha = \|y\|$. Thus, the dual problem (\ref{eq: proof-proj-dual-problem}) may be simplified to 
    \begin{equation}
        \min_{y \in\bbR^m} \left\lbrace \frac{1}{2}\|A^\top y\|^2 + (Ax-b)^\top y +\varepsilon  \|y\|   \right\rbrace. 
    \end{equation}    
    If the optimal dual variable $y^\star$ were zero, it would follow that $\|Au^\star -b\| = \|A(x + A^\top 0) -b\| = \|Ax-b\|>\varepsilon$, a contradiction to the constraint in (\ref{eq: proof-projection-problem}). Thus, $y^\star$ is nonzero and satisfies the first-order optimality condition  
    \begin{equation}
        0 
        = \frac{\mbox{d}}{\mbox{d}y} \left[ \frac{1}{2}\|A^\top y\|^2 + (Ax)^\top y +\varepsilon  \|y\| - y^\top b\right]_{y=y^\star}
        \hspace*{-12pt}
        = A A^\top y^\star + Ax+\frac{\varepsilon y^\star}{\|y^\star\|} -b
        = \left( A  A^\top + \frac{\varepsilon \II }{\|y^\star\|}\right) y^\star + Ax - b.
    \end{equation}
    By \ref{c: full-rank}, either $A$ has full row rank or $\varepsilon > 0$, and so the matrix $AA^\top + \varepsilon I /\|y^\star\|$ is invertible. Thus, 
    \begin{equation}
        y^\star = - \left(AA^\top + \frac{\varepsilon\II}{\|y^\star\|}\right)^{-1} (Ax-b).
        \label{eq: proof-ystar-recursive}
    \end{equation}
    Taking norms of both sides reveals 
    \begin{equation}
        \|y^\star\| =  \left\| \left(AA^\top + \frac{\varepsilon \II}{\|y^\star\|}\right)^{-1}(Ax-b)\right\|.
    \end{equation} 
    By Lemma \ref{lemma: tau-exists},
    there is a unique scalar $\tau_x > 0$ for which 
    \begin{equation}
        \frac{1}{\tau_x} = \left\| (AA^\top + \varepsilon \tau_x \II)^{-1} (Ax-b)\right\|.
    \end{equation}
    Hence $\|y^\star\| = 1/\tau_x$ and 
    (\ref{eq: proof-ystar-recursive}) becomes 
    \begin{equation}
        y^\star = - \left(AA^\top + \varepsilon\tau_x \II\right)^{-1} (Ax-b),
    \end{equation}
    from which we conclude
    \begin{equation}
        u^\star = x + A^\top y^\star
        = x - A^\top (AA^\top + \varepsilon \tau_x \II)^{-1} Ax,
    \end{equation}
    as desired.
    \end{proof}

\vspace*{0.5in}

We now verify our convergence result, which is a special case of existing results for Douglas Rachford splitting.  \\

{\bf Theorem \ref{thm: main-theorem}} (Convergence of PP). 
\textit{\mainTheorem}

\begin{proof}

    By \ref{c: proper-feasible},  
    $\mbox{ri}(\mbox{dom}(f)) \cap \mbox{ri}(\sC) \neq \varnothing$ and, in particular, $\mbox{dom}(f)\cap\sC \neq \varnothing$. This fact, together with \ref{c: coercive-bounded} and $\sC$ being closed and convex, enable
    \cite[Proposition 11.15]{bauschke2017convex} to be applied to deduce the existence of a solution to (\ref{eq: problem}). 
    Futhermore, by \cite[Proposition 6.19]{bauschke2017convex}, $0 \in \mbox{sri}(\sC - \mbox{dom}(f)),$\footnote{Here $\mbox{sri}(\cdot)$ denotes the strong relative interior (see \cite[Definition 6.9]{bauschke2017convex}).}  
    and so \cite[Proposition 27.8]{bauschke2017convex} may be applied to deduce 
    \begin{equation} 
        \{ x : 
        0 \in \partial f(x) + \partial N_{\sC}(x)\}
        =
        \argmin_{x\in\sC}  f(x)  
        = \argmin_{x\in\bbR^n} \Big\lbrace f(x) \, : \, \|Ax-b\|\leq\varepsilon\Big\rbrace,
        \label{eq: proof-thm-fonc}
    \end{equation}
    where $N_\sC$ is the normal cone operator for $\sC$. 
    By \ref{c: full-rank} and \ref{c: b-image} and Proposition \ref{prop: projection}, the update for each $x^k$ is precisely the projection of $z^k$ onto $\sC$.
    Thus, the iteration (\ref{eq: pp-iteration}) is an instance of Douglas-Rachford splitting \cite{bauschke2017convex,lions1979splitting}.  
    By \cite[Theorem 26.11]{bauschke2017convex}, $\{z^k\}$ and $\{x^k\}$ converge, and the limit $x^\star$ of $\{x^k\}$ satisfies 
    \begin{equation}
        0 \in \partial f(x^\star) + N_{\sC}(x^\star). 
        \label{eq: proof-thm-0-derivative}
    \end{equation}
    By (\ref{eq: proof-thm-fonc}) and (\ref{eq: proof-thm-0-derivative}), we conclude $x^\star$ is a solution to (\ref{eq: problem}), as desired.
    Lastly, note $\|Ax^k-b\|\leq\varepsilon$ since, as noted above, $x^k = \sP_\sC(z^k) \in \sC$ for all $k\in\bbN$.
\end{proof}

\newpage
\section{Numerical Examples Supplement} \label{app: experiments}

A subsection is dedicated herein to providing further details for each of the numerical examples, particularly formulations of the algorithms to which PP is compared and proofs for the special cases of the projection formula in Proposition \ref{prop: projection} to the various settings.

\subsection{Basis Pursuit} \label{app: bp}

We initialize iterates to the zero vector (\textit{e.g.} $z^1 = 0$ for PP). Entries of $A$ are drawn from $A_{ij}\sim \mathcal{N}(0,1/m)$.
In each case, we attempted to pick parameters that yield best performance while respecting conditions needed to ensure convergence guarantees. Here 10 trials were used, with the mean time reported in the main text and medians used for the plots in Figure \ref{fig: plots-bp}. \\

{\bf Proximal Projection (PP).} We applied Algorithm \ref{alg: pp-eps=0} with $\alpha = 0.1$ and the shrink operator.\\

{\bf Linearized Bregman (LB).}
Rather than directly minimize $\ell_1$, linearized Bregman solves
\begin{equation}
    \min_{x}  \mu \|x\|_1 + \frac{1}{2\alpha} \|u\|^2
    \quad\mbox{s.t.}\quad 
    Ax = b,
\end{equation}
which yields the same result as (\ref{eq: problem-bp}) when $\alpha$ is sufficiently large. 
Following \cite{cai2009linearized}, we use the iteration  
\begin{subequations}
    \begin{align}
        v^{k+1} & = v^k - A^\top (Ax^k - b) , \\ 
        x^{k+1} & =   \mbox{shrink}\left( \alpha v^{k+1},\ \alpha   \mu\right) .
    \end{align}
\end{subequations}
By \cite[Theorem 2.4]{cai2009linearized},  this iteration converges for $\alpha \in (0, 2/\|AA^\top\|)$; we used $\mu = 2 \|AA^\top\|$ and $\alpha = 2 / \|A A^\top \|$. \\

{\bf Linearized Method of Multipliers (LMM).} 
For step sizes $\alpha,\lambda > 0$, the linearized method of multipliers solves (\ref{eq: problem-bp}) using iterates of the form 
\begin{subequations}
    \begin{align}
        x^{k+1}
        & = \mbox{shrink}\left(x^k - \alpha A^\top [v^k + \lambda(Ax^k - b)],\   \alpha \right) , \\
        v^{k+1} & = v^k + \lambda (Ax^{k+1} - b).
    \end{align}
\end{subequations}
This converges for $\alpha \lambda \|A^\top A\| < 1$ (\textit{e.g.} see \cite[Section 3.5]{ryu2022large}). We used $\lambda = 100\|A^\top A\|$ and $\alpha = 1 / (\lambda \|A^\top A\|)$. \\ 

{\bf Primal Dual Hybrid Gradient (PDHG).}
For step sizes $\alpha,\lambda > 0$, the primal dual hybrid gradient (PDHG) iteration solves (\ref{eq: problem-bp}) using the iteration
\begin{subequations}
    \begin{align}
        x^{k+1} & = \mbox{shrink}(x^k - \alpha A^\top v^k,\ \alpha),\\ 
        v^{k+1} & =   v^k + \lambda  [ A(2x^{k+1} - x^k) - b] .
    \end{align}
    \label{eq: app-pdhg}\end{subequations}Here $\alpha \lambda \|A^\top A\| < 1$ ensures convergence (see \cite[Section 3.2]{ryu2022large}). We used $\lambda = 100 \|A^\top A\|$ and $\alpha = 1/(\lambda \|A^\top A\|)$.

\newpage
\subsection{Stable Principal Component Pursuit} \label{app: spcp}

For each method, we initialize the low rank term to $M$ and the sparse term to the zero matrix (\textit{e.g.} for proximal projection $X^1 = (X_L^1, X_S^1) = (M,0)$).

\paragraph*{Proximal Projection (PP).}
Letting $X = (X_L, X_S)$, the problem (\ref{eq: problem-spcp}) may be rewritten as 
\begin{equation}
    \min_{X} \|X_L\|_\star + \lambda \|X_S\|_1 
    \quad \mbox{s.t.}\quad 
    \|A X - M \|_F \leq \varepsilon,
\end{equation}
where $A = [\II \ \II]$. The proximal for (\ref{eq: problem-spcp}) can be written as
\begin{equation}
    \prox_{\alpha f}(X)
    = \left[\begin{array}{c} \prox_{\alpha \|\cdot\|_\star} (X_L) \\ \prox_{\alpha \lambda \|\cdot\|_1}(X_S)\end{array}\right]
    = \left[\begin{array}{c} \mbox{svt}(X_L, \ \alpha) \\ \mbox{shrink}_1(X_S,\ \alpha\lambda)\end{array}\right].
    \label{eq: spcp-prox}
\end{equation}

We next verify the projection formula used by PP for (\ref{eq: problem-spcp}). For ease of reference, the result is restated before its proof. \\

{\bf Lemma \ref{lemma: projection-spcp}} (SPCP Projection).
\textit{\lemmaProjectionSPCP} 
\begin{proof}
    When $Z$ is feasible, $\sP_\sC(Z) = Z$.
    In this case,  $\mu_Z = 0$ and so the result holds. 
    For the remainder of the proof, we assume $Z$ is not feasible.
    By Proposition \ref{prop: projection}, when $Z$ is not feasible, $\tau_Z$ is the positive solution to
    \begin{equation}
        \tau \left\| (AA^\top +   \varepsilon\tau \II)^{-1} [AZ-M] \right\|_F=1 
        \label{eq: app-spcp-tau-condition}.
    \end{equation}
    For each $\tau \geq 0$, note 
    \begin{equation}
        \left(AA^\top +  \tau \II \right)^{-1}
        = \left((2 +  \varepsilon\tau) \II \right)^{-1}
        = \frac{\II}{2+  \varepsilon\tau}.
    \end{equation}
    Thus, (\ref{eq: app-spcp-tau-condition}) may be rewritten as 
    \begin{equation}
    \frac{\tau}{2+ \varepsilon\tau} \|AZ - M\|_F = 1.
    \end{equation}
    Rearranging to isolate $\tau$ yields 
    \begin{equation}
        \tau_Z =  \frac{2}{\|AZ-M\|_F-\varepsilon}.
    \end{equation}
    Thus,  if $\|AZ-M \|_F > \varepsilon$, then 
    \begin{subequations}
    \begin{align}
        A^\top  (AA^\top + \varepsilon\tau_X \II)^{-1} (A Z -M)
        &= \frac{A^\top (AZ-M)}{2+\varepsilon \tau_Z} \\
        &= \frac{\|AZ-M\|_F - \varepsilon}{2\|AZ-M\|_F} A^\top (A Z -M)\\
        &= \frac{\|AZ-M\|_F - \varepsilon}{2\|AZ-M\|_F} \left[\begin{array}{c} Z_L + Z_S - M \\ Z_L + Z_S - M \end{array}\right]\\    
        &= \frac{\|Z_L+Z_S-M\|_F - \varepsilon}{2\|Z_L+Z_S-M\|_F} \left[\begin{array}{c} Z_L + Z_S - M \\ Z_L + Z_S - M\end{array}\right].
    \end{align}
    \end{subequations}
    Hence
    \begin{subequations}
    \begin{align}
        \sP_\sC(X) 
        & = 
        \begin{cases}
        \begin{array}{cl}
            X - \frac{\|X_L+X_S-M\|_F - \varepsilon}{2\|X_L+X_S-M\|_F} \left[\begin{array}{c} X_L + X_S - M \\ X_L + X_S - M\end{array}\right]
            & \mbox{if}\ \|X_L + X_S -M\|_F > \varepsilon \\ 
            X & \mbox{otherwise}
        \end{array}
        \end{cases} \\
        & = X - \max\left\lbrace \frac{\|X_L+X_S-M\|_F - \varepsilon}{2\|X_L+X_S-M\|_F} , 0 \right\rbrace  \left[\begin{array}{c} X_L + X_S - M \\ X_L + X_S - M\end{array}\right],
    \end{align}\label{eq: spcp-proj}\end{subequations}and
    the proof is complete.
\end{proof}

Substituting  the proximal formula (\ref{eq: spcp-prox}) and  projection formula from Lemma \ref{lemma: projection-spcp} into (\ref{eq: pp-iteration}) yields the updates 
\begin{subequations}
    \begin{align} 
        Z_L^{k+1} & =  \mbox{svt}(Z_L^k - 2\mu_k(Z_L^k+Z_S^k-M),\ \alpha)\\
        Z_S^{k+1} & =  \mbox{shrink}_1(Z_S^k - 2\mu_k(Z_L^k+Z_S^k-M),\ \lambda\alpha)
    \end{align}
\end{subequations}
Thus, Algorithm \ref{alg: pp} for (\ref{eq: problem-spcp}) simplifes to Algorithm \ref{alg: pp-spcp}. Furthermore, in the case where $\varepsilon = 0$, the sequence $\{Z^k\}$ is  the same as the sequence $\{X^k\}$ in (\ref{eq: spcp-pg}) below. \\

\paragraph*{Variant of Alternating Splitting Augmented Lagrangian Method (VASALM).}
For a parameter $\alpha > 0$, the augmented Lagrangian $\sL$ for (\ref{eq: problem-spcp}) is 
\begin{equation}
    \sL_\alpha (L, S, N; \Lambda)
    = \|L\|_\star + \lambda \|S\|_1 + \delta_{B(0,\varepsilon)}(N) + \left<\Lambda, L + S + N - M\right> + \frac{\alpha}{2}\|L + S + N - M\|_F^2.
\end{equation}
Note 
\begin{align}
    \frac{\alpha}{2} \left\| L + S + N - M +  \frac{\Lambda}{\alpha} \right\|_F^2
    &
    = \frac{1}{2}\left\|\sqrt{\alpha} (L + S + N - M ) + \frac{\Lambda}{\sqrt{\alpha}} \right\|_F^2\\
    & = \frac{\alpha}{2}\|L + S + N- M \|_F^2  + \left< \frac{\Lambda}{\sqrt{\alpha}},  \sqrt{\alpha}(L + S + N - M)\right> + \frac{1}{2\alpha}\|\Lambda\|_F^2 \\
    & = \frac{\alpha}{2}\|L + S + N - M \|_F^2  + \left<\Lambda, L + S + N - M\right> + \frac{1}{2\alpha}\|\Lambda\|_F^2,
\end{align} 
and so
\begin{equation*} 
    \sL_\alpha (L, S, N; \Lambda)
    = \|L\|_\star + \lambda \|S\|_1 + \delta_{B(0,\varepsilon)}(N) + \frac{\alpha}{2} \|L + S + N - M + \alpha \Lambda\|_F^2 - \frac{1}{2\alpha}\|\Lambda\|_F^2.
\end{equation*}
VASALM does proximal steps for $L$, $S$, and $N$ separately, with two dual variable $\Lambda$ updates. Specifically, for $\alpha >0$ and $\eta > 2$ it generates a sequence of updates via  
\begin{subequations}
    \begin{align}
        N^{k+1}
        & = \sP_{B(0,\varepsilon)}\left(    \frac{\Lambda^k}{\alpha} + M - X_L^k - X_S^k \right) \\ 
        \hat{\Lambda}^{k}
        & = \Lambda^k - \alpha(X_L^k + X_S^k + N^{k+1} - M) \\ 
        X_S^{k+1}
        & =  \mbox{shrink}\left( X_S^k + \frac{\hat{\Lambda}^k}{\alpha\eta}, \ \frac{\lambda}{\alpha\eta}\right)\\         
        X_L^{k+1}
        & = \mbox{svt}\left( X_L^k + \frac{\hat{\Lambda}^k}{\alpha\eta},\ \frac{1}{\alpha\eta}\right) \\
        \Lambda^{k+1}
        & = \hat{\Lambda}^k  + \alpha (X_L^{k} - X_L^{k+1}) + \alpha(X_S^{k} - X_S^{k+1}).
    \end{align}
\end{subequations}

We used $\eta = 3$ and $\alpha = 10^{-5}$.

\paragraph*{Partially Smoothed Proximal Gradient.}
Following the Nesterov smoothing technique \cite{nesterov2005smooth},
for a parameter $\mu > 0$, the work \cite{aybat2014efficient} considers the smoothed the version of (\ref{eq: problem-spcp}) given by 
\begin{equation}
    \min_{L,S} \lambda \|S\|_1  + \left( \max_{\|W\|\leq 1} \left<L,W\right> -\frac{\mu}{2}\|W\|_F^2 \right)  
    \quad \mbox{s.t.}\quad 
    \|L + S - M \|_F \leq \varepsilon,
    \label{eq: problem-spcp-smoothed}
    \tag{SPCP$_\mu$}
\end{equation}
which approaches (\ref{eq: problem-spcp}) as $\mu \rightarrow 0^+$.
The nuclear norm approximation $1/\mu$-smooth and has gradient given by 
\begin{equation}
    W_\mu(L) = U\, \mbox{diag}\left( \min\left\lbrace \frac{\sigma_i}{\mu},\ 1 \right\rbrace\right) V^\top,
\end{equation}
where $U\Sigma V^\top$ is here the SVD of $L$. 
Proximal gradient updates take the form 
\begin{equation}
    X^{k+1}
    = \argmin_{X} 
    \lambda \|S\|_1  + \left< W_\mu(X_L^k), \, X_L - X_L^k\right> + \frac{1}{2\mu} \| X_L - X_L^k\|_F^2
    \quad \mbox{s.t.}\quad 
    \|X_L + X_S - M\|_F \leq \varepsilon.
\end{equation}
Following  \cite[Lemma 6.1]{aybat2014efficient}, here proximal gradient update steps are explicitly given by 
\begin{subequations}
    \begin{align}
        X_S^{k+1} & = \mbox{shrink}\left( M - \mu (X_L^k - W_\mu(X_L^k)),  \frac{\lambda [1+\mu \theta^\star]}{\theta^\star} \right)\\
        X_L^{k+1} & = \frac{1}{1 + \mu\theta^\star} \Big[  \mu \theta^\star \left( M - X_S^{k+1}\right) + \big(X_L^k - \mu W_\mu(X_L^k) \big)\Big],
    \end{align} 
\end{subequations}
where $\theta^\star$ is the unique positive solution to\footnote{Here we assume $\varepsilon > 0$.}
\begin{equation}
    \varepsilon 
    = \left\| \min \left\lbrace \frac{\lambda}{\theta},\ \frac{|M-X_L^k + \mu W_\mu(X_L^k)|}{1+\mu \theta}\right\rbrace \right\|_F.
\end{equation}
Since $L - \mu W_\mu(L) 
= \mbox{svt}(L)$, the proximal gradient updates can be rewritten as 
\begin{subequations}
    \begin{align}
        X_S^{k+1} & = \mbox{shrink}\left( M - \mbox{svt}(X_L^k,\ \mu),  \frac{\lambda [1+\mu \theta^\star]}{\theta^\star} \right)\\
        X_L^{k+1} & = \frac{1}{1 + \mu\theta^\star} \Big[  \mu \theta^\star \left( M - X_S^{k+1}\right) +  \mbox{svt}(X_L^k, \ \mu)\Big]
    \end{align}
\end{subequations}
with $\theta^\star$ the solution to 
\begin{equation}
    \varepsilon 
    = \left\| \min \left\lbrace \frac{\lambda}{\theta},\ \frac{|M-\mbox{svt}(X_L^k,\ \mu)|}{1+\mu \theta}\right\rbrace \right\|_F.
\end{equation}
Since $M \in \bbR^{n_1\times n_2}$, we have the bound 
\begin{equation}
    \theta^\star \leq  \min\left\lbrace n_1 n_2  \lambda \varepsilon,\ \left|\frac{\|M-\mbox{svt}(X_L^k,\mu)\|_F - \varepsilon}{\mu\varepsilon}\right|\right\rbrace.
\end{equation}

By the comment following \cite[Theorem 2.1]{aybat2014efficient}, setting $\mu = \delta / \min\{n_1,n_2\}$ ensures a $\delta/2$-optimal solution to (\ref{eq: problem-spcp-smoothed}) is $\delta$-optimal solution to (\ref{eq: problem-spcp}). We set $\delta$ to be about $0.1$ times the optimal value for (\ref{eq: problem-spcp}) to ensure the objective of the limit is within $\sim$10\% of optimal. In Figure \ref{fig: plots-spcp}b, it appears that this choice of $\mu$ leaves a visible gap between the limit of PSPG and the optimal value. Reducing $\mu$ reduces the size of this gap, but also hinders the convergence rate of PSPG.

\paragraph*{Proximal Gradient (PG).} In the original work \cite{zhou2010stable} on SPCP, the authors follow the example of \cite{lin2009fast} to approximate (\ref{eq: problem-spcp}) in their numerical experiments by a soft-penalty variation  
\begin{equation}
    \min_{X} \|X_L\|_\star + \lambda \|X_S\|_1 + \frac{1}{2\mu}\|X_L + X_S - M\|_F^2.
    \label{eq: problem-spcp-soft}
\end{equation}
To apply proximal gradient, here we do the same. The update is 
\begin{equation}
    X^{k+1}
    = \prox_{\alpha f}\left( X^k - \frac{\alpha}{\mu} \left[\begin{array}{c} X_L^k + X_S^k - M \\ X_L^k + X_S^k - M\end{array}\right] \right), 
    \label{eq: spcp-pg}
\end{equation}
where the proximal is given in (\ref{eq: spcp-prox}). Here the quadratic term is $1/\mu$-smooth, and so a stepsize of $\alpha = \mu$ can be used to ensure (\ref{eq: spcp-pg}) converges to a solution to (\ref{eq: problem-spcp-soft}).
In this case, the iteration simplifies to 
\begin{subequations}
    \begin{align}
        X_{L}^{k+1}
        & = \mbox{svt}\left(M - X_S^k,\alpha\right) \\
        X_S^{k+1} 
        & = \mbox{shrink}\left(M - X_L^k,\alpha\lambda\right).
    \end{align}
\end{subequations}

\subsection{Earth Mover's Distance}\label{app: emd}

For simplicity, in the EMD example, we use $h=1.0$. The proof for the EMD projection is given below.
We emphasize, in this subsection, lowercase variables are represented in matrix form. Moreover, although  the divergence operator (denoted by $A$) is linear, its application in this form involves application of left and right matrix multiplications. To keep notation concise, we will write $A$ on the leftmost side with its application understood to be as described in the main text.\\

{\bf Lemma \ref{lemma: projection-emd}} (EMD Projection).
\textit{\lemmaProjectionEMD}\\

\begin{proof}
    This proof is a corollary of Proposition \ref{prop: projection}.
    For $z\in\sC$, the result directly follows from the proposition.
    In the remainder of the proof, we assume $z\notin \sC$. 
    The two tasks at at hand are to obtain a formula for the term multiplied by the matrix inverse in Propostion \ref{prop: projection} and to verify the choice for $\tau$ matches that in Proposition \ref{prop: projection}.
    Set 
    \begin{equation}
        q \triangleq [AA^\top + \varepsilon \tau \II]^{-1}(Am + \rho^1 - \rho^0)
    \end{equation}
    so that, by Proposition \ref{prop: projection} and (\ref{eq: emd-A-operation}),
    \begin{equation}
        \sP_\sC(z) = z - A^\top q = z - \left[ \begin{array}{c} K^\top q^1 \\ K^\top (q^2)^\top \end{array}\right].
    \end{equation}
    It follows that 
    \begin{equation}
        [AA^\top + \varepsilon\tau\II] q = Am + \rho^1 - \rho^0.
    \end{equation}    
    By the formulas in  (\ref{eq: emd-A-operation}),  
    \begin{equation} 
        [AA^\top + \varepsilon\tau \II] q
        = \left[ KK^\top  + \frac{\varepsilon\tau}{2}\II\right] q + q \left[ KK^\top + \frac{\varepsilon\tau}{2}\II\right] .
    \end{equation}    
    For $U\Sigma V^\top$ the SVD of $K$, direct multiplication reveals
    \begin{equation}
        KK^\top + \frac{\varepsilon\tau}{2}\II  
        = U \mbox{diag}(\sigma_i^2) U^\top + \frac{\varepsilon\tau}{2}\II  
        = U \, \mbox{diag}\left(\sigma_i^2 + \frac{\varepsilon\tau}{2}\right)   U^\top,
    \end{equation}
    where the final equality holds by the orthogonality of $U$. 
    Set $Y = U^\top q  U$ so that 
    \begin{equation}
        U\, \mbox{diag}\left(\sigma_i^2 + \frac{\varepsilon\tau}{2}\right)  Y U^\top 
        + U Y\, \mbox{diag}\left(\sigma_i^2 + \frac{\varepsilon\tau}{2}\right)  U^\top 
        = [AA^\top + \varepsilon\tau \II]q
        = Am + \rho^1 - \rho^0.
        \label{eq: app-emd-proj-proof-1}
    \end{equation}
    Left multiplying each term by $U^\top$ and then right multiplying by $U$, (\ref{eq: app-emd-proj-proof-1}) becomes   
    \begin{equation}
        \mbox{diag}\left(\sigma_i^2 + \frac{\varepsilon\tau}{2}\right)  Y
        + Y  \mbox{diag}\left(\sigma_i^2 + \frac{\varepsilon\tau}{2}\right)  = U^\top (Am + \rho^1 - \rho^0) U.
        \label{eq: app-emd-proj-proof-2}
    \end{equation}
    In element-wise form, (\ref{eq: app-emd-proj-proof-2}) may be equivalently written as
    \begin{equation}
        \left(\sigma_{i}^2 +  \frac{\varepsilon\tau}{2}\right)Y_{ij}
        + Y_{ij} \left(\sigma_j^2 + \frac{\varepsilon\tau}{2}\right)
        = (U^\top (Am + \rho^1 - \rho^0) U)_{ij},
        \quad \mbox{for all}\ i,j \in[n].
    \end{equation}
    Thus, 
    \begin{equation}
        Y_{ij} 
        = \frac{(U^\top (Am + \rho^1 - \rho^0)U)_{ij}}{\sigma_i^2 + \sigma_j^2 + \varepsilon\tau},
    \end{equation}
    where the division is well-defined since \ref{c: full-rank} ensures $\sigma_i^2 + \sigma_j^2 + \varepsilon\tau > 0$ for all $\tau > 0$.
    This verifies the formula for $Y_\tau$ in the lemma statement.
    Lastly, note the orthogonality of $U$ ensures
    \begin{equation}
       \| q \|_F =  \|UY_\tau U^\top \|_F 
       = \| Y_\tau U^\top \|_F 
       = \|Y_\tau \|_F,
    \end{equation}     
    which verifies the condition on $\tau$ may be expressed as $1 = \tau \|Y_\tau\|_F$.        
\end{proof}

\paragraph*{Proximal Projection (PP).} We used Algorithm \ref{alg: pp-emd} with $\alpha = 10^{-4}$. 

\paragraph*{Primal Dual Hybrid Gradient (PDHG).} 
We use the iteration in (\ref{eq: app-pdhg}) with $\lambda = 5$ and  $\alpha = 1/(5 \|A^\top A\|)$.

\paragraph*{G-Prox PDHG.}
Using the Hodge decomposition, the flux $m$ can be decomposed as $m = u + \nabla \psi$, where $u$ is a divergence free vector field and $\nabla \psi$ is a gradient field for which  
\begin{equation}
    \mbox{div}(\nabla \psi) + \rho^1 - \rho^0 = 0.
\end{equation}
In this example, we used
\begin{equation}
    \nabla \psi = \sP_{\sC}(0). 
\end{equation}
With $m = u + \nabla \psi$, the EMD problem can be rewritten as
\begin{equation}
    \min_{u} \| u + \nabla \psi \|_1 
    \quad\mbox{s.t.}\quad 
    \mbox{div}(u) = 0.  
\end{equation} 
To solve this, for parameters $\sigma > 0$ and $\tau > 0$, \cite{jacobs2019solving} proposes the iteration\footnote{Here we compress (3.14), (3.15), and (3.16) from that work.}
\begin{subequations}
    \begin{align}
        p^{n+1} & = 
        \frac{p^n + \sigma [u^n + \nabla \psi]}{\max\{ 1,  |p^n + \sigma [u^n + \nabla \psi|]\}}
        \\
        u^{n+1} & = u^n - \tau \sP_{\nabla^\perp}(2p^{n+1} - p^n),
    \end{align}
\end{subequations}
where $\nabla^\perp$ is the set of divergence free vector fields.
To compute this projection, we use Lemma \ref{lemma: projection-emd}, but with $\sC = \{ m : \|\mbox{div}(m) \|\leq \varepsilon\}$. 
This iteration is guaranteed to converge for $\sigma \tau < 1$; we used $\tau = 10^{-4}$ and $\sigma = 10^4$.

\newpage

\subsection{Stable Matrix Completion} \label{app: smc}

We begin with notation. In vectorized form of $X$, the matrix $A$ for (\ref{eq: problem-smc}) would be a diagonal matrix with $mn$ rows and $mn$ columns.
The $ij$-th entry on the diagonal would be $1$ if $(i,j)\in\Omega$ and $0$ otherwise. Consequently, $A = A^\top$ and $A = \sP_\Omega$.
With this choice for $A$, we may prove the projection formula. 

{\bf Lemma \ref{lemma: projection-mc}} (SMC Projection). 
\textit{\lemmaProjectionSMC}\\
\begin{proof}
    Consider a matrix $X$. 
    By Proposition \ref{prop: projection}, if $X\in\sC$, then $\sP_\sC(X) = X$.
    In what remains, we assume $X\notin\sC$.
    By Proposition \ref{prop: projection}, the projection is given by 
    \begin{equation}
        \sP_\sC(X) 
        = X - A^\top (AA^\top + \varepsilon\tau_X \II)(AX - \sP_\Omega(M)).
        = X - \sP_\Omega \left( \left[ \sP_\Omega + \varepsilon\tau_X\II\right]^{-1}\left[ \sP_\Omega(X - M)\right] \right),
    \end{equation}
    where again we note  $A=A^\top = \sP_\Omega$ and $\sP_\Omega^2 = \sP_\Omega$.
    Set 
    \begin{equation}
        Y \triangleq 
        (\sP_\Omega + \varepsilon\tau_X \II)^{-1} \sP_\Omega(X-M).
    \end{equation} 
    This implies
    \begin{equation}
        \sP_\Omega(Y) + \varepsilon \tau_X Y = \sP_\Omega(X-M)
        \quad\iff\quad 
        Y_{ij} = 
        \begin{cases}
            \begin{array}{cl}
                \frac{X_{ij}-M_{ij}}{1+\varepsilon\tau_X} & \mbox{if}\ (i,j)\in\Omega,\\
                0 & \mbox{otherwise},
            \end{array}
        \end{cases} 
        \quad \mbox{for all}\ (i,j)\in[m]\times[n].
    \end{equation}
    Consequently, 
    \begin{equation}
        \sP_\Omega(X) = X - \sP_\Omega(Y) 
        = X - \sP_\Omega\left(\frac{X-M}{1+\varepsilon \tau_X}\right)
        = \sP_{\Omega^\perp}(X) + \sP_\Omega\left( \frac{M + \varepsilon \tau_X X}{1+\varepsilon\tau_X}\right).
        \label{eq: app-proof-smc-projection-1}
    \end{equation}
    By Proposition \ref{prop: projection}, $\tau_X$ is the unique positive scalar satisfying 
    \begin{equation}
        1 = \tau_X \|  Y \|_F 
        = \tau_X  \left\| \sP_\Omega\left( \frac{X-M}{1+\varepsilon\tau_X}\right)\right\|_F
        = \frac{\tau_X}{1+\varepsilon\tau_X} \| \sP_\Omega(X-M)\|_F, 
    \end{equation}
    and so rearranging reveals 
    \begin{equation}
        \tau_X = \frac{1}{\|\sP_\Omega(X-M)\|_F - \varepsilon}.
    \end{equation}
    Plugging   this choice for $\tau_X$ into the projection formula (\ref{eq: app-proof-smc-projection-1}) reveals 
    \begin{equation}
        \sP_\Omega(X) 
        = \sP_{\Omega^\perp} (Z) + \sP_\Omega\left( \frac{[\|\sP_\Omega(Z-M)\|_F - \varepsilon] M + \varepsilon Z}{\|\sP_\Omega(Z-M)\|_F} \right) 
    \end{equation}
    as desired.
\end{proof}

{\bf Smoothed Proximal Gradient (SPG).}
Following \cite{aybat2014efficient}, for a parameter $\mu > 0$ we consider the smoothed problem
\begin{equation}
    \min_{X}   \left( \max_{\|W\|\leq 1} \left<X,W\right> -\frac{\mu}{2}\|W\|_F^2 \right)  
    \quad \mbox{s.t.}\quad 
    \|\sP_\Omega(X - M) \|_F \leq \varepsilon,
    \label{eq: problem-smc-smoothed}
    \tag{SMC$_\mu$}
\end{equation}
In much the same fashion as the stable principal component pursuit problem in Appendix \ref{app: spcp}, here the  smooth proximal gradient updates can be rewritten as  
    \begin{align} 
        X^{k+1} =  \sP_\Omega\left(\frac{1}{1 + \mu\theta^\star} \Big[  \mu \theta^\star  M  +  \mbox{svt}(X^k, \ \mu)\Big] \right) + \sP_{\Omega^\perp} \left( \mbox{svt}(X^k,\ \mu) \right),
    \end{align} 
with  
\begin{equation}
    \varepsilon 
    = \left\|  \frac{\sP_\Omega(M-\mbox{svt}(X^k,\ \mu))}{1+\mu \theta^\star} \right\|_F
    \quad\iff \quad 
     \theta^\star 
    = \frac{\| \sP_\Omega(M - \mbox{svt}(X^k,\mu))\|_F - \varepsilon}{ \mu \varepsilon}
\end{equation}
Here we use $\mu = 10$. As noted before, reducing $\mu$ reduces the sub-optimality gap of the limit of SPG, but reduces the smoothness and, thus, step-size (which hinders convergence rate).

\paragraph*{Variant of Alternating Splitting Augmented Lagrangian Method (VASALM).}
Similar to (\ref{eq: problem-spcp}), for a parameter $\alpha > 0$, the augmented Lagrangian $\sL$ for (\ref{eq: problem-smc}) is 
\begin{equation}
    \sL_\alpha (X, N; \Lambda)
    = \|X\|_\star   + \delta_{B(0,\varepsilon)}(\sP_\Omega(N)) + \left<\Lambda, \sP_\Omega(X + N - M)\right> + \frac{\alpha}{2}\|\sP_\Omega(X + N - M)\|_F^2.
\end{equation}
\def\sK{\mathcal{K}}
Define the set
\begin{equation}
    \sK \triangleq \{ Y : \|\sP_\Omega(Y)\|_F \leq \varepsilon\}.
\end{equation}
The projection $\sP_\sK$ onto $\sK$ is simply a Euclidean projection onto the $\varepsilon$-ball about the origin for the submatrix $Y_\Omega$ and the remainder of $Y$ is unchanged by the projection.  
The VASALM algorithm in this context uses the iterates 
\begin{subequations}
    \begin{align}
        N^{k+1}
        & = P_{\sK}\left( \frac{\Lambda^k}{\alpha} + \sP_\Omega(M) - X  \right)\\
        \hat{\Lambda}^k 
        & = \Lambda^k - \alpha \left( X^k + N^k -\sP_\Omega( M^k)\right) \\ 
        X^{k+1} 
        & = \mbox{svt}\left(X^k + \frac{\hat{\Lambda}^k}{\eta\alpha},\ \frac{1}{\eta\alpha}\right) \\ 
        \Lambda^{k+1}
        & = \hat{\Lambda}^k + \alpha (X^k - X^{k+1}),
    \end{align}
\end{subequations}
where $\eta > 2$. We use $\eta = 3$ and $\alpha = 10^{-2}$.
 
\end{document}